\documentclass{amsart}

\usepackage[paperheight=11in, 
    paperwidth=8.5in,
    outer=1.15in,
    inner=1.15in,
    bottom=1.15in,
    top=1.15in]{geometry} 

\usepackage{amsmath, amsxtra, amsthm, amsfonts, amssymb, mathtools}
\usepackage{mathrsfs}
\usepackage{graphicx}
\usepackage{url}
\usepackage{color}
\usepackage{bbm}
\usepackage{tikz-cd}

\usepackage[T1]{fontenc}
\usepackage{enumitem}

 \usepackage[hidelinks]{hyperref} 
 \hypersetup{
    colorlinks,
    citecolor=magenta,
    filecolor=magenta,
    linkcolor=blue,
    urlcolor=black
}

\usetikzlibrary{positioning}
\usetikzlibrary{matrix}
\usetikzlibrary{decorations}
\usetikzlibrary{decorations.pathreplacing, decorations.pathmorphing, angles,quotes}
 
\newtheorem{maintheorem}{Theorem}

\newtheorem{theorem}{Theorem}[section]
\newtheorem{proposition}[theorem]{Proposition}
\newtheorem{lemma}[theorem]{Lemma}
\newtheorem{corollary}[theorem]{Corollary}

\theoremstyle{definition}
\newtheorem{definition}[theorem]{Definition}
\newtheorem{remark}[theorem]{Remark}
\newtheorem{example}[theorem]{Example}

\newcommand{\RR}{\mathbb R}
\newcommand{\CC}{\mathbb C}
\newcommand{\PP}{\mathbb P}
\newcommand{\ZZ}{\mathbb Z}
\newcommand{\M}{\mathrm{M}}
\newcommand{\be}{\mathbf{e}}
\newcommand{\kk}{\mathbbm{k}}

\newcommand{\indeg}{\operatorname{ind}} 
\newcommand{\rk}{\operatorname{rk}}

\title{Vanishing theorems for combinatorial geometries} 

\author{Christopher Eur}
\author{Alex Fink}
\author{Matt Larson}

\begin{document}

\maketitle

\vspace{-1.5\baselineskip}

\begin{abstract}
We establish strong vanishing theorems for line bundles on wonderful varieties of hyperplane arrangements, and we show that the resulting positivity properties of Euler characteristics extend to all matroids.  We achieve this by showing that every degeneration of a wonderful variety within the permutohedral toric variety is reduced and Cohen--Macaulay.  The same holds for a larger class of subschemes in products of projective lines that we call ``kindred,'' which are characterized by matroidal Hilbert polynomials.  We establish positivity properties for $K$-rings of matroids.  Our results give a new proof of the nonnegativity of the omega invariant of a matroid, in support of Speyer's $f$-vector conjecture, and resolve the conjecture of Toh\u{a}neanu that higher order Orlik--Terao algebras are Cohen--Macaulay.
\end{abstract}

\section{Introduction}\label{sec:introduction}

For a nonnegative integer $n$, let $[n] = \{1, \dotsc, n\}$.
For a subset $S\subseteq [n]$, let $\be_S := \sum_{i\in S} \be_i 
\in \RR^{n}$ denote the sum of standard basis vectors indexed by $S$.
A \emph{matroid} $\M$ of \emph{rank $r$} on $[n]$ is a nonempty collection $\mathcal B$ of  subsets of ${[n]}$ of cardinality $r$, called the set of \emph{bases} of $\M$, satisfying the property:
\begin{quote}
\text{For every edge of the polytope $P(\M) := \text{the convex hull of } \{\be_B : B\in \mathcal B\} \subset \RR^{n}$},\\
\text{\qquad\qquad\qquad}\text{there is a pair $\{i\neq j\} \subseteq {[n]}$ such that the edge is parallel to $\be_i - \be_j$.}
\end{quote}
The polytope $P(\M)$ is called the \emph{base polytope} of~$\M$.
We say that $\M$ is \emph{loopless} if every $i\in [n]$ is contained in a basis of~$\M$.
We point to \cite{Wel76, Oxl11} as standard references on matroid theory 
and to \cite{GGMS87} for the equivalence of the definition of matroids given here with theirs.

\smallskip
The prototypical example of a matroid is given by a linear subspace $L \subseteq \kk^{[n]}$ over a field $\kk$, which defines a matroid of rank $r=\dim L$ with set of bases
\[
\{B \subseteq {[n]} : \text{the composition of $L \hookrightarrow \kk^{[n]}$ with the projection $\kk^{[n]} \twoheadrightarrow \kk^B$ is an isomorphism}\}.
\]
Matroids that arise in this way are called \emph{realizable matroids}, and they provide an interface between matroid theory and geometry.
A key medium of this interface occurs through the \emph{wonderful variety} $W_L$ of $L\subseteq \kk^{[n]}$, introduced by de Concini and Procesi \cite{dCP95}.
To avoid trivialities, we assume that $L$ is not contained in a coordinate hyperplane of~$\kk^{[n]}$, or, equivalently, that the matroid of $L$ is loopless.
Then $W_L$ is constructed as follows.
Write $\mathbb PV$ for the projective space of lines in a vector space $V$.
Define a rational map 
\[
\phi \colon \PP(\kk^{[n]}) \dashrightarrow \prod_{i< j \in [n]} \PP(\kk^{\{i,j\}})=(\PP^1)^{\binom n2},
\]
where each factor $\PP(\kk^{[n]}) \dashrightarrow \PP(\kk^{\{i,j\}})$ is the rational map induced by the coordinate projection $\kk^{[n]} \to \kk^{\{i,j\}}$.
Then the wonderful variety $W_L$ is the closure of $\phi(\PP L)$ in $(\PP^1)^{\binom n2}$, which is smooth and can be described as an iterated blow-up of $\mathbb{P}L$ \cite[Proposition 1.6]{dCP95} (see the discussion following Proposition~\ref{prop:strong normality} for the agreement of the definition in~\cite{dCP95} with ours).
In particular $W_L$ is contained in the closure of the image of $\phi$,
which we denote $X_{[n]}$ and call the \emph{permutohedral variety (of dimension $n-1$\/)}.

\smallskip
While almost all matroids are not realizable \cite{Nel18}, geometric properties of realizable matroids sometimes persist for all matroids.
A landmark example arises from the intersection theory of~$W_L$; because it is a smooth projective variety, its intersection numbers display positivity properties. 
This was the template for the development of the Hodge theory of matroids \cite{AHK18},
which established the corresponding positivity properties for Chow rings of matroids.
Here, we prove cohomological vanishing results for line bundles on~$W_L$,
and we show that the positivity properties of sheaf Euler characteristics implied by these vanishing results
extend to all matroids.

\smallskip
The line bundles in question are the pullbacks to $W_L$ of nef line bundles on $X_{[n]}$. 
Because $X_{[n]}$ is a smooth projective toric variety, these have the following combinatorial description;
we point to \cite{CLS11} as a standard reference on toric geometry and adopt its conventions.
Applying to $X_{[n]}$ the standard dictionary between toric nef divisors and polytopes \cite[Chapter 6]{CLS11},
one obtains a correspondence between nef toric divisor classes on~$X_{[n]}$ and \emph{generalized permutohedra} (up to translation), which are lattice polytopes $P \subset \RR^n$ satisfying the property that every edge of $P$ is parallel to $\be_i - \be_j$ for some $i\neq j \in [n]$. (See \cite[Section 2.7]{BEST23} and \cite{ACEP20} for details of this correspondence, including the validity of working with polytopes in an affine translate of the character lattice $\{x \in \ZZ^n : x_1 + \dotsb + x_n = 0\}$ of the open torus of $X_{[n]}$).
For example, base polytopes of matroids are exactly generalized permutohedra which are contained in the unit cube.
Let $\mathcal L_P$ denote the corresponding nef line bundle on $X_{[n]}$.
We also write $\mathcal L_P$ for the restriction of the line bundle to $W_L$. 
Nef line bundles on projective toric varieties are basepoint-free, and the complete linear system of $\mathcal{L}_P$ induces a map $f_P \colon X_{[n]} \to \PP(\kk^{P\cap \ZZ^n})$, where the set of lattice points $P \cap \ZZ^n$ is identified with the torus-eigenbasis of $H^0(X_{[n]},\mathcal L_P)$ \cite[Chapter 4.3]{CLS11}.

\begin{maintheorem}\label{thm:vanish}
Let $P$ be a generalized permutohedron in $\RR^n$.  
Then:
\begin{enumerate}[label = (\arabic*)]
\item\label{item:vanish} $H^i(W_L, \mathcal{L}_P^{\otimes a}) = 0$ unless either: \textnormal{(i)} $a\geq 0$ and $i = 0$, or  \textnormal{(ii)} $a<0$ and $i = \dim f_P(W_L)$.
\item\label{item:surject} For any $a \ge 0$, the restriction map $H^0(X_{[n]}, \mathcal{L}_P^{\otimes a}) \to H^0(W_L, \mathcal{L}_P^{\otimes a})$ is surjective. 
\item\label{item:push} We have $R{f_P}_* \mathcal{O}_{W_L} \simeq \mathcal{O}_{f_P(W_L)}$. 
\end{enumerate}
\end{maintheorem}

In Section~\ref{sec:examples}, we show that the conclusion of Theorem~\ref{thm:vanish} fails for some other classes of nef or ample line bundles on wonderful varieties.
In this sense, Theorem~\ref{thm:vanish} appears to be sharp.
An explicit combinatorial formula for $\dim f_P(W_L)$ is provided in Theorem~\ref{thm:numericaldim}.  
If $P$ is of the maximal dimension, i.e., ($n-1$)-dimensional, then $\dim f_P(W_L) = \dim W_L = r-1$.  We note the following corollary.

\begin{corollary}\label{cor:ACM}
The section ring $\bigoplus_{a \ge 0} H^0(W_L, \mathcal{L}_P^{\otimes a})$ coincides with the homogeneous coordinate ring of $f_P(W_L)$ in $\PP(\kk^{P \cap \ZZ^n})$, and in particular is generated in degree $1$.  The subvariety $f_P(W_L)$ in $\PP(\kk^{P \cap \ZZ^n})$ is projectively normal and arithmetically Cohen--Macaulay.
\end{corollary}

\begin{proof}
Because a generalized permutohedron $P$ is a normal lattice polytope (Proposition~\ref{prop:strong normality}),
the homogeneous coordinate ring of $\PP(\kk^{P \cap \ZZ^n})$ surjects onto the section ring $\bigoplus_{a\geq 0} H^0(X_{[n]},\mathcal L_P^{\otimes a})$.  Therefore \ref{item:surject} implies that the section ring $\bigoplus_{a\geq 0} H^0(W_L, \mathcal L_P^{\otimes a})$ is the homogeneous coordinate ring of the subvariety $f_P(W_L)$ in $\PP(\kk^{P \cap \ZZ^n})$.
Then \ref{item:push} implies that the section ring $\bigoplus_{a\geq 0} H^0(W_L, \mathcal L_P^{\otimes a})$ coincides with the section ring $\bigoplus_{a\geq 0} H^0(f_P(W_L), \mathcal O(a))$ of  $\mathcal O(1)$ on $\PP(\kk^{P \cap \ZZ^n})$.
Because $W_L$ is smooth (and hence normal), the isomorphism $\mathcal O_{f_P(W_L)}\simeq {f_P}_*\mathcal O_{W_L}$ from \ref{item:push} implies that $f_P(W_L)$ is normal, and hence projectively normal. The projection formula and \ref{item:push} imply that the natural map from $H^i(f_P(W_L), \mathcal{O}(a))$ to $H^i(W_L, \mathcal{L}_P^{\otimes a})$ is an isomorphism for all $i$ and $a$. 
The cohomology vanishing in \ref{item:vanish} implies that $f_P(W_L)$ is arithmetically Cohen--Macaulay: see, e.g., \cite[Proposition 4.3]{EL}. 
\end{proof}

Special cases of this corollary recover results or resolve conjectures in the prior literature.
\begin{itemize}
\item When $P$ is the convex hull of $\{\be_S : S\subset [n] \text{ and } |S|= n-1\}$, the homogeneous coordinate ring of $f_P(W_L)$ is known as the \emph{Orlik--Terao algebra} \cite{OrlikTerao,OTAlgebra}, and Corollary~\ref{cor:ACM} recovers a main result of \cite{PS06} that this algebra is Cohen--Macaulay.  

\item More generally, when $P$ is the convex hull of $\{\be_S : S\subset [n] \text{ and } |S|= n-k\}$ for $1\leq k \leq n-1$, the homogeneous coordinate ring of $f_P(W_L)$ is known as the 
\emph{higher order Orlik--Terao algebra}, studied extensively in the commutative algebra literature \cite{OTCM,GHM,Burity}.  In this case, Corollary~\ref{cor:ACM} resolves and generalizes the conjecture of  Toh\u{a}neanu that the higher order Orlik--Terao algebra for $k = 2$ is Cohen--Macaulay \cite{OTafold}.
\item When $P = -P(\M)$, the convex hull of $\{-\be_{B} : B \text{ a basis of the matroid $\M$ of $L$}\}$, the variety $f_P(W_L)$ is known as \emph{Kapranov's visible contours compactification}, appearing in the study of moduli of hyperplane arrangements \cite{Kap93,HKT06,KTChow}.
This variety
is the log-canonical model of a hyperplane arrangement complement associated to the data of $L\subseteq \kk^{[n]}$.
Corollary~\ref{cor:ACM} states that this variety is arithmetically Cohen--Macaulay.
\end{itemize} 

\medskip
We now describe the numerical positivity of sheaf Euler characteristics implied by Theorem~\ref{thm:vanish} that we will establish for all matroids.
Corollary~\ref{cor:ACM} implies that the polynomial $a \mapsto \chi(W_L, \mathcal{L}_P^{\otimes a})$ is the Hilbert function of a graded Cohen--Macaulay algebra which is generated in degree $1$. The Hilbert functions of such algebras were classified by Macaulay: 
for a univariate polynomial $p(a)$ of degree $d$, if one writes 
\[
\sum_{a \ge 0} p(a) t^a = \frac{h_0^* + h_1^*t + \dotsb + h_d^*t^d}{(1 - t)^{d+1}},
\]
then the polynomial $p$ is the Hilbert function of a graded Cohen--Macaulay algebra that is generated in degree $1$ if and only if $(h_0^*, \dotsc, h_d^*)$ is a \emph{Macaulay vector} (also known as an $M$-vector or $O$-sequence).
This means that $h_i^* \ge 0$ for each $i$ and that the sequence $(h_0^*, \dotsc, h_d^*)$ satisfies certain explicit inequalities bounding how quickly it can grow; see \cite[Theorem 4.2.10]{BH93} for details. Note that $(h_0^*, \dotsc, h_d^*)$ are determined by the equation
$$p(a) = \sum_{i=0}^{d} h_i^* \binom{a + d - i}{d}.$$
In particular, we see that $h_d^* = (-1)^d p(-1)$, and that $(-1)^d p(a) \ge 0$ for $a < 0$. 

\medskip
Previous work allows us to state an analogue of the above numerical positivity for all matroids, as we now describe.
The proof of \cite[Corollary 10.6]{BEST23} gave an explicit combinatorial formula for the $K$-class $[\mathcal O_{W_L}]$ in the Grothendieck group $K_\circ(X_{[n]})$ of coherent sheaves on $X_{[n]}$, showing that the class depends only on the matroid of $L$.
In particular, for any given $\mathcal L\in \operatorname{Pic}(X_{[n]})$, by the projection formula for $W_L\hookrightarrow X_{[n]}$, the quantity $\chi(W_L, \mathcal L|_{W_L}) = \chi(X_{[n]}, \mathcal O_{W_L} \otimes \mathcal L)$ depends only on the matroid realized by $L$.
This suggested the possibility of generalizing the map $\chi(W_L,{-})$ to a function whose first argument is a (not necessarily realizable) matroid.
The authors of \cite{LLPP24} carried out this program, defining the \emph{matroid $K$-ring} $K(\M)$ of a loopless matroid $\M$, equipped with a surjection from the Grothendieck ring $K^\circ(X_{[n]})$ of vector bundles on $X_{[n]}$ to $K(\M)$
and an Euler characteristic map $\chi(\M,{-}) \colon K(\M)\to\mathbb Z$,
so that, when $L$ realizes $\M$, there is an isomorphism $K(\M)\cong K^\circ(W_L)$ identifying the functions $\chi(\M,{-})$ and $\chi(W_L,{-})$.
Then, 
the authors of \cite{EL} defined the \emph{$h^*$ vector} $h^*(\mathrm{M}, \mathcal{L})$ of a loopless matroid $\mathrm{M}$ and a line bundle $\mathcal L \in \operatorname{Pic}(X_{[n]})$ via the equality
$$\sum_{a \ge 0} \chi(\mathrm{M}, \mathcal{L}^{\otimes a}) t^a = \frac{h_0^*(\M,\mathcal L) + h_1^*(\M,\mathcal L)t + \dotsb + h_d^*(\M,\mathcal L)t^d}{(1 - t)^{d+1}},$$
where $d$ is the degree of the polynomial $a \mapsto \chi(\mathrm{M}, \mathcal{L}^{\otimes a})$.
When $\mathrm{M}$ is realizable and $\mathcal{L}$ is the line bundle associated to a generalized permutohedron, we have explained how Corollary~\ref{cor:ACM} implies that $h^*(\mathrm{M}, \mathcal{L})$ is a Macaulay vector.
We show this holds more generally, proving \cite[Conjecture 4.8]{EL}.

\begin{maintheorem}\label{thm:Macaulay}
Let $\mathrm{M}$ be a loopless matroid on $[n]$, and let $P$ be a generalized permutohedron in $\mathbb{R}^n$. Then, the $h^*$ vector $h^*(\mathrm{M}, \mathcal{L}_P)$ is a Macaulay vector. 
\end{maintheorem}

An explicit combinatorial formula for the degree $d$ of the polynomial $a\mapsto \chi(\M, \mathcal L_P^{\otimes a})$ is given in Theorem~\ref{thm:numericaldim}.
If $P$ is of maximal dimension $(n-1)$,  then $d = \operatorname{rank}(\M) - 1$.

\medskip
A consequence of Theorem~\ref{thm:Macaulay} is a new proof of the nonnegativity of the omega invariant of a matroid, in support of the 20-year-old $f$-vector conjecture of Speyer \cite{SpeThesis}, as we now explain.
If $P$ is a generalized permutohedron, then the negated polytope $-P = \{-p : p \in P\}$ is as well.
Define the invariant $\omega(\M)\in\mathbb Z$ of a loopless matroid $\M$ of rank~$r$ by
\[
\omega(\M) = (-1)^{r-n + \dim P(\M)}\chi(\M, \mathcal L_{-P(\M)}^{-1}).
\]
The $\omega$ invariant is the leading coefficient of Speyer's invariant $g_{\M}(t)\in\mathbb Z[t]$, which was defined in \cite{Spe09, FS12} in an attempt to bound the complexity of polyhedral subdivisions of the base polytope $P(\M)$ of a matroid $\M$ into base polytopes of matroids.
Such subdivisions had arisen in the study of Grassmannians and moduli spaces of hyperplane arrangements \cite{Laf03, Kap93, HKT06}; their more recent appearances are surveyed in \cite[Section 4.4]{Ard18}. 
Speyer conjectured an upper bound on the number of faces of each dimension of such a subdivision, the titular $f$-vector.
He reduced the conjecture to showing that the coefficients of $g_{\M}(t)$ are nonnegative (the manifestation in this context of $K$-theoretic positivity),
and he proved this in \cite{Spe09} for matroids realizable over fields of characteristic zero.
Our result shows the nonnegativity of the omega invariant, which was first proved in the recent preprint \cite{BF}.

\begin{corollary} \label{cor:speyer}
For any loopless matroid $\M$, one has
$\omega(\M) \geq 0$.
\end{corollary}

If $\M$ is realizable over a field $\kk$ of characteristic zero and $\dim P(\M) = n-1$, then $\mathcal{L}_{-P(\M)}$ is nef and big, so this follows immediately from the Kawamata--Viehweg vanishing theorem. 

\begin{proof}[Proof of Corollary~\ref{cor:speyer}]
Because $P(\M)$ and $-P(\M)$ induce the same partition of~$[n]$ in the sense of Section~\ref{sec:dimensions},
Corollary~\ref{cor:numericaldim coarsening} implies that the polynomial $a\mapsto\chi(\M,\mathcal L_{-P(\M)}^{\otimes a})$ has degree $d:=r-n+\dim P(\M)$.
Then $(-1)^d\chi(\M,\mathcal L_{-P(\M)}^{-1}) = h^*_d(\M,\mathcal L_{-P(\M)})$,
which is nonnegative by Theorem~\ref{thm:Macaulay}.
\end{proof}

We sketch the proof of the main theorems in the case when $\mathcal L_P$ is ample on~$X_{[n]}$.
If one can show that the wonderful variety $W_L$ degenerates to a reduced and Cohen--Macaulay union of torus-invariant strata in $X_{[n]}$, then the vanishing results in Theorem~\ref{thm:vanish} follow from Frobenius splitting techniques.
In fact, we produce a collection of $K$-classes on~$X_{[n]}$ such that every subscheme having one of these $K$-classes is reduced and Cohen--Macaulay: we call these subschemes ``kindred.''
We prove that $W_L$ is kindred. Since $W_L$ always has some torus-invariant degeneration, which shares its $K$-class, Theorem~\ref{thm:vanish} follows. As we have explained, this also gives Theorem~\ref{thm:Macaulay} when $\M$ is realizable.
In this way, we avoid the difficult task of explicitly controlling the combinatorics of degenerations of~$W_L$. 
To extend Theorem~\ref{thm:Macaulay} to all matroids, the task remaining is to obtain a kindred subscheme of~$X_{[n]}$ whose $K$-class is the one appropriate to compute $\chi(\M,{-})$. 
We construct a union of torus-invariant subvarieties with the desired property, which we call the ``tropical initial degeneration of~$\M$,'' using the tropical degeneration formula of Katz \cite[Theorem 10.1]{Kat09}.

\subsection*{Organization}
In Section~\ref{sec:kindred}, we define and exhibit properties of kindred subschemes. There we work in the setting of subschemes of $(\mathbb P^1)^\ell$. Applications to our main theorems are obtained via the embedding $X_{[n]}\hookrightarrow(\mathbb P^1)^{\binom n2}$ and coordinate projections thereof.
In Section~\ref{sec:vanish}, we construct the tropical initial degeneration, relate it to the wonderful variety and to matroid Euler characteristics, and prove our main theorems. 
In Section~\ref{sec:dimensions}, we give a formula for the degree of the polynomial $a\mapsto \chi(\M, \mathcal L_P^{\otimes a})$.
In Section~\ref{sec:discussion}, we collect examples that support the optimality of Theorem~\ref{thm:vanish}, and we establish similar vanishing theorems for close cousins of wonderful varieties.

\subsection*{Previous work}
When $P$ is a Minkowski sum of standard simplices, Theorem~\ref{thm:Macaulay} was proved in \cite{EL}.
However, this class of polytopes does not include many of the most interesting cases, such as the ones listed below Corollary~\ref{cor:ACM}.

In \cite{Ruizhen}, Liu extends the techniques of \cite{BF22} to show that,
if $\operatorname{char}\kk = 0$ and $P$ is a Minkowski sum of matroid polytopes, then
$H^i(W_L, \mathcal L_P) = 0$ for $i > 0$, proving part of Theorem~\ref{thm:vanish}\ref{item:vanish} in this case.

Corollary~\ref{cor:speyer} was also proved in \cite{BF} by a completely different method. The authors control the sign of Euler characteristics of the form $\chi(X_{[n]}, [\bigwedge^i \mathcal{Q}_{\M_1}^{\vee}] \otimes [\bigwedge^j \mathcal{Q}_{\M_2}^{\vee}])$ for a pair of matroids $\M_1, \M_2$ on~$[n]$, where $\mathcal{Q}_{\M}$ is one of the tautological classes of \cite{BEST23}. It follows from \cite[Proposition 5.3]{LLPP24} (see also \cite[Corollary 10.6]{BEST23}) that
$$\chi(\M_1, \mathcal{L}_{-P(\M_2)}^{-1}) = \chi(\M_1, [\textstyle{\bigwedge^{n-r}} \mathcal{Q}_{\M_2}^{\vee}]) = \sum_{i \ge 0} (-1)^i\chi(X_{[n]}, [\textstyle{\bigwedge^i} \mathcal{Q}_{\M_1}^{\vee}] \otimes [\textstyle{\bigwedge^{n-r}} \mathcal{Q}_{\M_2}^{\vee}]).$$
When $\M_1 = \M_2 = \M$, \cite[Lemma 6.2]{FS12} gives that $\chi(X_{[n]}, [\bigwedge^i \mathcal{Q}_{\M}^{\vee}] \otimes [\bigwedge^{n-r} \mathcal{Q}_{\M}^{\vee}]) = 0$ unless $i=r$, and so $\omega(\M) = (-1)^{\dim P(\M) - n}\chi(X_{[n]}, [\bigwedge^{r} \mathcal{Q}_{\M}^{\vee}] \otimes [\bigwedge^{n-r}\mathcal{Q}_{\M}^{\vee}])$. Using this formula, the positivity result \cite[Theorem D]{BF} then implies Corollary~\ref{cor:speyer}. However, the techniques of \cite{BF} cannot be used to prove either Theorem~\ref{thm:vanish} or Theorem~\ref{thm:Macaulay}, nor can the results in this paper be used to deduce the positivity consequences of \cite[Theorem D]{BF}.

\subsection*{Acknowledgments}
We thank Jenia Tevelev and Ronnie Cheng for showing us examples similar to Example~\ref{ex:grid},
Louis Esser for helping us understand the geometry of Example~\ref{ex:fano},
Andy Berget for helpful comments on an earlier draft, and Michel Brion for helpful conversations.
We thank the Mathematisches 
Forschungsinstitut Oberwolfach for hosting the 2025 Toric Geometry workshop where this paper began.
The first author is supported by US NSF DMS-2246518. 
The second author received support from the Engineering and Physical Sciences Research Council [grant number EP/X001229/1],
as well as the Institute for Advanced Study and the Fields Institute for Research in Mathematical Sciences. The third author is supported by the Charles Simonyi Endowment and the Oswald Veblen Fund at the Institute for Advanced Study. 
We thank the referee for careful reading and helpful suggestions.

\subsection*{AI Declaration} No AI was used for this work.

\section{Kindred subschemes}\label{sec:kindred}

We begin by introducing a remarkable family of connected components of the Hilbert scheme of a product of projective lines. 
Subschemes in these components, which we call kindred subschemes, have a very close relationship to their degenerations. 
Kindred subschemes are a special case of subschemes defined by Cartwright--Sturmfels ideals \cite{CDNG15,CDNG20} (Remark~\ref{rem:CS}). 

A set is \emph{independent} in a matroid $\M$ if it is a subset of a basis of~$\M$.

\begin{definition}\label{def:kindred}
A closed subscheme $X$ of $(\mathbb{P}^1)^{\ell}$ is said to be \emph{kindred} if it is empty, or if there is a matroid $\mathrm{M}$ on $[\ell]$ such that, for all $(a_1, \dotsc, a_{\ell}) \in \mathbb{Z}^{\ell}$, we have
$$\chi(X, \mathcal{O}(a_1, \dotsc, a_{\ell})) = \sum_{I \text{ independent in }\mathrm{M}} \,\, \prod_{i \in I} a_i.$$
\end{definition}

We call the matroid $\mathrm{M}$ appearing in Definition~\ref{def:kindred} the \emph{progenitor matroid} of $X$. 
It is clear that whether a subscheme $X$ is kindred only depends on the class $[\mathcal{O}_X]$ in the Grothendieck group of coherent sheaves $K_{\circ}((\mathbb{P}^1)^{\ell})$. 
The function $(a_1, \dotsc, a_{\ell}) \mapsto \chi(X, \mathcal{O}(a_1, \dotsc, a_{\ell}))$ is a polynomial known as the \emph{Snapper polynomial} of~$X$, after \cite{Sna59}.
It agrees with the $\mathbb{Z}^\ell$-graded Hilbert polynomial of any ideal $I$ determining $X$ in the homogeneous coordinate ring $\kk[(\mathbb{P}^1)^\ell]=\kk[x_1, \dotsc, x_\ell, y_1, \dotsc, y_{\ell}]$, i.e., for $a_1,\ldots,a_\ell$ sufficiently large, we have
\[\chi(X, \mathcal{O}(a_1, \dotsc, a_{\ell})) = \dim_{\kk}(\kk[(\mathbb{P}^1)^\ell]/I)_{(a_1,\ldots,a_\ell)}.\]

Recall that the fundamental class $[Y]$ of a subscheme $Y$ of $(\mathbb{P}^1)^{\ell}$ is the class in $A_{\dim Y}((\mathbb{P}^1)^{\ell})$ given by summing the classes of the top-dimensional irreducible components of $Y$, with coefficient equal to the multiplicity of that component. More generally, if $\mathcal{F}$ is a coherent sheaf on $(\mathbb{P}^1)^{\ell}$, then its fundamental class is the sum of the classes of the top-dimensional irreducible components of its support, with coefficient equal to the rank of the sheaf at the generic point of that component. 

For a subset $S$ of $[\ell]$, let $Y_S = \{(x_1, \dotsc, x_{\ell}) : x_i = 0 \text{ if }i \not \in S\} \subset (\mathbb{P}^1)^{\ell}$. Then $\{[Y_S]\}_{S \subseteq [\ell]}$ is a basis for $A_{*}((\mathbb{P}^1)^{\ell})$.
By the next proposition, an application of the Hirzebruch--Riemann--Roch theorem, we can recover the fundamental class of a sheaf from its Snapper polynomial.

\begin{proposition}\label{prop:HRR}\cite[Example 15.2.16 \& Example 15.1.5]{Ful98}
Let $\mathcal{F}$ be a coherent sheaf on $(\mathbb{P}^1)^{\ell}$ whose support has dimension $d$ and whose fundamental class is $\sum_{S \in \binom{[\ell]}{d}} c_S [Y_S]$. Then the polynomial $(a_1, \dotsc, a_{\ell}) \mapsto \chi((\mathbb{P}^1)^{\ell}, \mathcal{F}(a_1, \dotsc, a_{\ell}))$ has degree $d$, and its degree $d$ part is $\sum_{S \in \binom{[\ell]}{d}} c_S \prod_{i \in S} a_i$. 
\end{proposition}
\noindent
Applying Proposition~\ref{prop:HRR} when $\mathcal{F} = \mathcal{O}_X$ for some kindred $X \subseteq (\mathbb{P}^1)^{\ell}$, we have the following formula. 

\begin{corollary}\label{prop:fund}
Let $X$ be a kindred subscheme of $(\mathbb{P}^1)^{\ell}$ with progenitor matroid $\mathrm{M}$. Then we have
$$[X] = \sum_{B \text{ basis of }\mathrm{M}} [Y_B] \in A_{\dim X}((\mathbb{P}^1)^{\ell}).$$
\end{corollary}

\begin{proposition}\label{prop:matroidexample}
Let $\mathrm{M}$ be a matroid on $[\ell]$, and let $ Y_{\mathrm{M}} = \bigcup_{B \text{ basis}} Y_B$. Then $Y_{\mathrm{M}}$ is a kindred subscheme of $(\mathbb{P}^1)^{\ell}$ with progenitor matroid $\mathrm{M}$. 
\end{proposition}

\begin{proof}
This is a special case of \cite[Proposition 2.14]{EL}, but we sketch an easy direct proof.
The homogeneous coordinate ring of $Y_{\mathrm{M}}$ is the quotient of $\kk[x_1, \dotsc, x_{\ell}, y_{1}, \dotsc, y_{\ell}]$ by the ideal $\langle \prod_{i \in D} y_i : D\subseteq [\ell] \text{ not independent in }\M\rangle$.
The $\mathbb{Z}^{\ell}$-graded Hilbert function of this ring is given by the formula in Definition~\ref{def:kindred} since, for $\boldsymbol a\in \ZZ_{\geq 0}^\ell$, the set $\{\prod_{i\in [\ell]} x_i^{a_i - b_i}y_i^{b_i} : \{i : b_i >0\} \text{ is independent in $\M$}\}$ forms a monomial basis for the degree $\boldsymbol a$ component of the quotient ring.
Because the Hilbert function agrees with the Snapper polynomial (i.e., the multigraded Hilbert polynomial) on points of $\mathbb{Z}^{\ell}$ which are sufficiently deep in the positive orthant, this implies the result. 
\end{proof}

\begin{proposition}\label{prop:kindredprops}
Let $X$ be a kindred subscheme of $(\mathbb{P}^1)^{\ell}$ with progenitor matroid $\mathrm{M}$. Then $X$ degenerates to $Y_{\mathrm{M}}$ inside of $(\mathbb{P}^1)^{\ell}$.  
\end{proposition}

\begin{proof}
Let $B$ be the $\ell$-fold product of the group of $2 \times 2$ lower triangular matrices acting factorwise on $(\mathbb{P}^1)^{\ell}$, with unique fixed point given by $(0, \dotsc, 0)$. Then $B$ acts on the Hilbert scheme of $(\mathbb{P}^1)^{\ell}$. Because $B$ is solvable, by the Borel fixed point theorem, the $B$-orbit of the point corresponding to $X$ in the Hilbert scheme has a $B$-fixed point in its closure. This gives a degeneration of $X$ to a subscheme $Z$ of $(\mathbb{P}^1)^{\ell}$ which is fixed by $B$. 

We have $[\mathcal{O}_X] = [\mathcal{O}_Z]$ in $K_{\circ}((\mathbb{P}^1)^{\ell})$, and therefore $[X] = [Z] \in A_{\dim X}((\mathbb{P}^1)^{\ell})$.  The reduction $Z^{\mathrm{red}}$ is a union of $B$-fixed subvarieties of $(\mathbb{P}^1)^{\ell}$. The only $B$-fixed subvarieties of $(\mathbb{P}^1)^{\ell}$ are the $Y_S$, for various $S$. 
These have linearly independent classes in $A_{\dim X}((\mathbb{P}^1)^{\ell})$,
so in order to have $[X] = [Z]$, Corollary~\ref{prop:fund} shows that $Z$ must contain $Y_{\mathrm{M}}$, where $\mathrm{M}$ is the progenitor matroid of~$X$. 
By Proposition~\ref{prop:matroidexample}, $Y_{\mathrm{M}}$ is kindred, and so $[\mathcal{O}_{Y_\M}] = [\mathcal{O}_X] = [\mathcal{O}_Z]$. In particular, $Y_{\mathrm{M}} = Z$. 
\end{proof}

\begin{corollary}\label{cor:CM}
Let $X$ be a kindred subscheme of $(\mathbb{P}^1)^{\ell}$. Then $X$ is Cohen--Macaulay, pure dimensional, geometrically reduced, and geometrically connected. 
\end{corollary}

\begin{proof}
Let $\mathrm{M}$ be the progenitor matroid of $X$. It is clear that $Y_{\mathrm{M}}$ is pure dimensional, geometrically reduced, and geometrically connected.
Furthermore, it is known that $Y_{\mathrm{M}}$ is Cohen--Macaulay: the multihomogeneous coordinate ring of $Y_\M$ is equal to (the extension by extra variables of) the Stanley--Reisner ring of the matroid independence complex, which is Cohen--Macaulay \cite[Theorem 3.2.1]{PB80}.
Therefore the same properties are true for $X$ \cite[\href{https://stacks.math.columbia.edu/tag/045U}{Tag 045U}, \href{https://stacks.math.columbia.edu/tag/02NM}{Tag 02NM}, \href{https://stacks.math.columbia.edu/tag/0C0D}{Tag 0C0D}, \href{https://stacks.math.columbia.edu/tag/055J}{Tag 055J}]{stacks-project}.
\end{proof}

The next corollary of Proposition~\ref{prop:kindredprops} uses the theory of Frobenius splitting; see \cite{BK05} for a comprehensive reference on Frobenius splittings. If $\kk$ has characteristic $p > 0$ and $V$ is a scheme over $\kk$, then a \emph{Frobenius splitting} is a splitting of the map $\mathcal{O}_V \to F_* \mathcal{O}_V$ induced by the absolute Frobenius $F$. A closed subscheme $Y$ of $V$ is \emph{compatibly split} if the image of $F_* \mathcal{I}_Y$ under this splitting is contained in $\mathcal{I}_Y$; here $\mathcal{I}_Y$ is the ideal sheaf of $Y$ in $V$. 
Given an effective Cartier divisor $D$ on $V$, a \emph{Frobenius splitting relative to $D$} is a map $F_* \mathcal{O}_V(D) \to \mathcal{O}_V$ such that the composition $F_* \mathcal{O}_V \to F_*\mathcal{O}_V(D) \to \mathcal{O}_V$ is a Frobenius splitting; $Y$ is compatibly split relative to $D$ if no irreducible component of $Y$ is contained in the support of $D$ and it is compatibly split by the composition $F_* \mathcal{O}_V \to F_*\mathcal{O}_V(D) \to \mathcal{O}_V$. 

\begin{lemma}\label{lem:splittingprop}
Suppose that $\kk$ has positive characteristic, and let $V$ be a projective scheme over $\kk$. Suppose $V$ is Frobenius split and a closed subscheme $Y$ of $V$ is compatibly split. Then 
\begin{enumerate}[label = (\arabic*)]
\item\label{item:1} For any ample line bundle $\mathcal{L}$ on $V$, we have $H^i(Y, \mathcal{L}) = 0$ for $i > 0$, and the restriction map $H^0(V, \mathcal{L}) \to H^0(Y, \mathcal{L})$ is surjective.
\item \label{item:2} If $Y$ is Cohen--Macaulay and pure dimensional and $\mathcal{L}$ is an ample line bundle, then $H^i(Y, \mathcal{L}^{-1}) = 0$ for $i < \dim Y$. 
\item \label{item:3} If $V$ is Frobenius split relative to an ample divisor $D$, $Y$ is compatibly split relative to $D$, and $\mathcal{L}$ is a globally generated line bundle on $V$, we have $H^i(Y, \mathcal{L}) = 0$ for $i > 0$, and the restriction map $H^0(V, \mathcal{L}) \to H^0(Y, \mathcal{L})$ is surjective.
\end{enumerate}
\end{lemma}

\begin{proof}
The first statement is \cite[Theorem 1.2.8]{BK05}, the second statement is \cite[Theorem 1.2.9]{BK05}, and the third statement is \cite[Theorem 1.4.8]{BK05}. 
\end{proof}

\begin{corollary}\label{prop:kindredvanishing}
Let $X$ be a kindred subscheme of $(\mathbb{P}^1)^{\ell}$. Then, for any $a_1, \dotsc, a_{\ell} \ge 0$, we have 
$$H^i(X, \mathcal{O}(a_1, \dotsc, a_{\ell})) = 0 \text{ for }i > 0,$$
and the restriction map 
$$H^0((\mathbb{P}^1)^{\ell}, \mathcal{O}(a_1, \dotsc, a_{\ell})) \to H^0(X, \mathcal{O}(a_1, \dotsc, a_{\ell}))$$ 
is surjective. 
For any $a_1, \dotsc, a_{\ell} < 0$, we have
$$H^i(X, \mathcal{O}(a_1, \dotsc, a_{\ell})) = 0 \text{ for }i < \dim X.$$
\end{corollary}
\begin{proof}
By upper semicontinuity \cite[\href{https://stacks.math.columbia.edu/tag/0BDN}{Tag 0BDN}]{stacks-project}, it suffices to prove these statements for $Y_{\mathrm{M}}$. For the surjectivity of the restriction map, one applies upper semicontinuity to the kernel of the surjective map $H^0((\mathbb{P}^1)^{\ell}, \mathcal{O}(a_1, \dotsc, a_{\ell})) \otimes \mathcal O_X \to \mathcal{O}_X(a_1, \dotsc, a_{\ell})$ of vector bundles.

When $(\mathbb{P}^1)^{\ell}$ is viewed as a homogeneous space for $(GL_2)^{\ell}$,
$Y_{\mathrm{M}}$ is a union of Schubert varieties. In particular, $Y_{\mathrm{M}}$ is defined over $\operatorname{Spec} \mathbb{Z}$, and the reduction of $Y_{\mathrm{M}}$ modulo any prime $p$ is compatibly Frobenius split in $(\mathbb{P}^1)^{\ell}$ relative to an ample divisor \cite[Theorem 2.2.5, Proposition 1.2.1]{BK05}. The line bundle $\mathcal{O}(a_1, \dotsc, a_{\ell})$ is defined over the prime field $\mathbb F$ of $\kk$, and we can check the cohomology vanishing and the surjectivity of the restriction map over $\mathbb F$ \cite[\href{https://stacks.math.columbia.edu/tag/02KH}{Tag 02KH}]{stacks-project}.
If the characteristic of $\kk$ is positive, then the result follows from \ref{item:2} and \ref{item:3} in Lemma~\ref{lem:splittingprop} and the Cohen--Macaulayness of $Y_\M$; see also  \cite[Section 4]{Bri03}. If $\kk$ has characteristic $0$, then the desired result follows from upper semicontinuity. 
\end{proof}

We will make use of the following standard consequence of the Leray spectral sequence and Serre vanishing. See, for example, \cite[Lemma 2.1]{Hyry}. 

\begin{proposition}\label{prop:leray}
Let $f \colon X \to Y$ be a proper morphism, and let $\mathcal{L}$ be an ample line bundle on $Y$. Then $R^if_* \mathcal{O}_X = 0$ for all $i > 0$ if and only if $H^i(X, f^* \mathcal{L}^{\otimes n}) = 0$ for all $i > 0$ and all $n$ sufficiently large.
\end{proposition}

\begin{proposition}\label{prop:projection}
Let $X$ be a kindred subscheme of $(\mathbb{P}^1)^{\ell}$, and let $p \colon (\mathbb{P}^1)^{\ell} \to (\mathbb{P}^1)^{m}$ be a coordinate projection. Then $p(X)$ is a kindred subscheme of $(\mathbb{P}^1)^{m}$ and $Rp_* \mathcal{O}_X \simeq \mathcal{O}_{p(X)}$. 
\end{proposition}

\begin{proof}
We may assume that $p$ is the projection onto the first $m$ coordinates. 
 By Corollary~\ref{prop:kindredvanishing}, the higher cohomology on $X$ of positive twists of $\mathcal{O}(1, \dotsc, 1, 0, \dotsc, 0) = p^* \mathcal{O}(1, \dotsc, 1)$ vanishes. By Proposition~\ref{prop:leray}, we have $R^ip_* \mathcal{O}_X = 0$ for $i > 0$. 

We proceed to show that the natural map from $\mathcal{O}_{p(X)}$ to $p_*\mathcal{O}_X$ is an isomorphism. Because the higher direct images vanish, we have
\begin{equation}\label{eq:projection 1}
\chi((\mathbb{P}^1)^m, p_* \mathcal{O}_X \otimes \mathcal{O}(a_1, \dotsc, a_m)) = \chi(X, \mathcal{O}(a_1, \dotsc, a_m, 0, \dotsc, 0))
\end{equation}
for any $a_1, \dotsc, a_m$. Let $\mathrm{M}$ be the progenitor of $X$, and let $\mathrm{N}$ denote the restriction of $\mathrm{M}$ to $[m]$, i.e., $\mathrm{N}$ is the matroid whose independent sets are the independent sets of $\mathrm{M}$ which are contained in~$[m]$. 
By considering the leading term of the polynomial on the right-hand side of \eqref{eq:projection 1}, 
Proposition~\ref{prop:HRR} implies that the fundamental class of $p_* \mathcal{O}_X$ is $\sum_{B \text{ basis of }\mathrm{N}} [Y_B]$. 
Because the support of $p_*\mathcal{O}_X$ is $p(X)$, this implies that the fundamental class of $p(X)$ is $\sum_{B \text{ basis of }\mathrm{N}} [Y_B]$. 

As in the proof of Proposition~\ref{prop:kindredprops}, we can degenerate $p(X)$ to a subscheme $Z$ which contains $Y_{\mathrm{N}} \coloneqq \bigcup_{B \text{ basis of }\mathrm{N}}Y_B$. Let $\mathcal{I}$ be the ideal sheaf of $Y_{\mathrm{N}}$ in $Z$. We have 
$$[\mathcal{O}_{p(X)}] - [\mathcal{O}_{Y_{\mathrm{N}}}] = [\mathcal{I}] \in K_{\circ}((\mathbb{P}^1)^m).$$
On the other hand, because there is an injective map $\mathcal{O}_{p(X)} \to p_* \mathcal{O}_X$
we have $[\mathcal{O}_{p(X)}] = [p_* \mathcal{O}_X] - [p_* \mathcal{O}_X/\mathcal{O}_{p(X)}]$. By comparing the Snapper polynomials and using Proposition~\ref{prop:matroidexample}, we have that $[p_* \mathcal{O}_X] = p_* [\mathcal{O}_X] = [\mathcal{O}_{Y_{\mathrm{N}}}]$. We deduce that
$$-[p_* \mathcal{O}_X/\mathcal{O}_{p(X)}] = [\mathcal{I}].$$
For $n$ sufficiently large, we have $\chi((\mathbb{P}^1)^m, p_* \mathcal{O}_X/\mathcal{O}_{p(X)} \otimes \mathcal{O}(n, \dotsc, n)) \ge 0$, with equality if and only if $p_* \mathcal{O}_X/\mathcal{O}_{p(X)} = 0$. Similarly, $\chi((\mathbb{P}^1)^m, \mathcal{I}\otimes \mathcal{O}(n, \dotsc, n)) \ge 0$, with equality if and only if $\mathcal{I} = 0$. We deduce that $p_*\mathcal{O}_X = \mathcal{O}_{p(X)}$ and $\mathcal{O}_Z = \mathcal{O}_{Y_{\mathrm{N}}}$, so $p(X)$ is kindred. 
\end{proof}

\begin{remark}
It is to obtain Proposition~\ref{prop:projection} that we have included the matroid condition in the definition of kindred subschemes.
Proposition \ref{prop:kindredprops} and its corollaries 
hold true for the broader class of subschemes whose Snapper polynomial is $\sum_{I\in\mathcal S}\prod_{i\in I}a_i$ 
for any family $\mathcal S\subseteq 2^{[\ell]}$ of sets closed under taking subsets, all of whose maximal elements have the same cardinality:
let us call such families $\mathcal S$ \emph{pure}.
In the proof of Proposition~\ref{prop:projection}, we see that the fundamental class of the projection of $X$ to~$(\mathbb{P}^1)^J$ is a sum over the maxima of the set family $\{I\in\mathcal S:I\subseteq J\}$.
Requiring this set family to be pure for all $J\subseteq[\ell]$
forces $\mathcal S$ to be the independent set family of a matroid (or empty) \cite[Exercise 1.1.3]{Oxl11}.
In any event, the matroid condition loses us no generality for integral subschemes (Proposition~\ref{prop:integralBrion}).
\end{remark}

We now give examples of kindred subschemes and discuss related results in the literature. The most important example for us will be Corollary~\ref{cor:kindredinM}, to come. Many other examples arise from the following result of Brion \cite{Bri03} (see also \cite[\S4.4]{Per14} for an exposition in the context of spherical varieties).

\begin{proposition}\label{prop:integralBrion}
Let $X$ be an integral subvariety of $(\mathbb{P}^1)^{\ell}$ whose fundamental class is multiplicity-free, i.e., the coefficients in the expansion of $[X]$ in terms of the $[Y_S]$ are all either $0$ or $1$. Then $X$ is kindred. 
\end{proposition}

\begin{proof}
By \cite[Theorem 4.6]{BH20} or \cite{CCrLMZ}, the fundamental class of $X$ is equal to the fundamental class of $Y_{\mathrm{M}}$ for some matroid $\mathrm{M}$. By \cite{Bri03}, $X$ degenerates to a reduced union of the $Y_S$. As the fundamental class is preserved under this degeneration, $X$ degenerates to $Y_{\mathrm{M}}$, which is kindred by Proposition~\ref{prop:matroidexample}. 
The degeneration also preserves the class in $K_{\circ}((\mathbb{P}^1)^{\ell})$, so $X$ is kindred.
\end{proof}

\begin{example}\label{ex:WL}
The wonderful variety $W_L$, embedded in $(\mathbb{P}^1)^{\binom{n}{2}}$ as in Section~\ref{sec:introduction}, has a fundamental class which is multiplicity-free \cite{BinglinLi}, and so $W_L$ is a kindred subscheme.  This also follows directly from \cite[Corollary 7.5]{LLPP24}, which computes the Snapper polynomial of~$W_L$:
see the proof of Corollary~\ref{cor:kindredinM} below for more detail.
As a consequence, any subscheme $Z$ of $X_{[n]}$ with $[\mathcal{O}_Z] = [\mathcal{O}_{W_L}]$ in $K_{\circ}(X_{[n]})$ defines a kindred subscheme of $(\mathbb{P}^1)^{\binom{n}{2}}$, and in particular is reduced and Cohen--Macaulay. 
\end{example}

Many further examples of subvarieties of $(\mathbb{P}^1)^{\ell}$ whose fundamental class is multiplicity-free can be obtained by taking the closure in $(\mathbb{P}^1)^{\ell}$ of certain subvarieties of $\mathbb{A}^{\ell}$. 
Closures of linear subspaces \cite{AB16} and of certain determinantal varieties \cite{HL} are multiplicity-free, and so give examples of kindred subschemes. A fact central to~\cite{BF} is that if $L_1$ and $L_2$ are linear subspaces of $\kk^n$ such that $L_2$ is not contained in any coordinate hyperplane, then the closure of the semi-inverted Hadamard product
$\{v_1 \cdot v_2^{-1} : v_1 \in L_1, \, v_2 \in L_2 \cap (\kk^*)^n\}$
in $(\mathbb{P}^1)^n$ is multiplicity-free.

\begin{remark}\label{rem:CS}
Kindred subschemes of $(\mathbb{P}^1)^{\ell}$ are closely related to \emph{Cartwright--Sturmfels} ideals in the $\mathbb{N}^{\ell}$-graded ring $\kk[(\mathbb{P}^1)^{\ell}] = \kk[x_1, \dotsc, x_\ell, y_1, \dotsc, y_{\ell}]$, the homogeneous coordinate ring of $(\mathbb{P}^1)^{\ell}$ \cite{CDNG15,CDNG20}. These are, by definition, homogeneous ideals whose multigraded generic initial ideal is radical. It is a consequence of Corollary~\ref{prop:kindredvanishing} and Proposition~\ref{prop:kindredprops} that the saturated ideal in $\kk[(\mathbb{P}^1)^{\ell}]$ defining a kindred subscheme is a Cartwright--Sturmfels ideal. Conversely, the subscheme of $(\mathbb{P}^1)^{\ell}$ which is defined by a Cartwright--Sturmfels ideal whose multidegree is the bases of a matroid is a kindred subscheme. However, the $K$-theoretic nature of the definition of kindred subschemes is more convenient for our purposes. 
\end{remark}

\begin{remark}
Finally, we remark that in \cite{Brion} (see also \cite[Theorem 4.3]{BF22}), Brion showed that an integral subvariety of $(\mathbb{P}^1)^{\ell}$ whose multidegree is multiplicity-free is normal and, if the characteristic of $\kk$ is $0$, has rational singularities. This complements Proposition~\ref{prop:projection}, as in some cases the singularities of a kindred subvariety can be resolved using a projection from a kindred subvariety embedded in a larger product of projective lines. 
In this manner Proposition \ref{prop:projection} can be used to replace the invocations of rational singularities in the collapsing arguments of \cite{BF22},
thus allowing the characteristic~0 assumption of that work to be removed.
\end{remark}

\section{Vanishing theorems}\label{sec:vanish}
In this section we introduce a subscheme of the permutohedral variety $X_{[n]}$ for each loopless matroid $\M$,
the \emph{tropical initial degeneration} $\indeg_w\M$. It depends on the choice of a ``sufficiently generic'' weight $w$. When $\mathrm{M}$ is realized by $L \subseteq \kk^n$, the subscheme $\indeg_w\M$ will be a Gr\"{o}bner degeneration of $W_L$ inside of $X_{[n]}$. Recall that there is an embedding of $X_{[n]}$ into $(\mathbb{P}^1)^{\binom{n}{2}}$.
We will show that $\indeg_w \M$ is a kindred subscheme of $(\mathbb{P}^1)^{\binom{n}{2}}$, and we will use this to prove Theorems \ref{thm:vanish} and~\ref{thm:Macaulay}.

\smallskip
Let $\Sigma_{[n]}$ be the \emph{permutohedral fan}, the fan of the toric variety $X_{[n]}$,
so that cones $\sigma$ of $\Sigma_{[n]}$ index closed boundary strata $V(\sigma)$ of~$X_{[n]}$.
To each loopless matroid $\M$ on~$[n]$, there is an associated subfan $\Sigma_\M$ of~$\Sigma_{[n]}$ called the \emph{Bergman fan} of~$\M$ \cite{AK06}.
Concretely, $\Sigma_{[n]}$ is the rational fan on $\mathbb Z^n/\mathbb Z\be_{[n]}$
whose rays are generated by the $\be_S$ for all nonempty proper subsets $S\subset[n]$,
and in which $\be_{S_1},\ldots,\be_{S_k}$ lie in a common cone if and only if, up to reordering, $S_1\subset\cdots\subset S_k$.
The \emph{rank} $\rk_\M(S)$ of a set $S\subseteq[n]$ in the matroid $\M$ is the size of a largest intersection of $S$ with a basis of~$\M$,
and a \emph{flat} of~$\M$ is a subset of~$[n]$ maximal with respect to its rank.
Then $\Sigma_\M$ contains exactly the cones of $\Sigma_{[n]}$ whose ray generators all correspond to flats of~$\M$.
The dimension of $\Sigma_\M$ is $r-1$, where $r$ is the rank of~$\M$.

Throughout, choose some $w \in \mathbb{Q}^n/\mathbb{Q}\be_{[n]}$ which is sufficiently generic, in the sense that for each pair of cones $\sigma$ and~$\tau$ of $\Sigma_{[n]}$,
if the intersection $\sigma\cap(w+\tau)$ is nonempty, then it is transverse, i.e., $\dim \sigma \cap (w + \tau) = \dim \sigma + \dim \tau - (n-1)$. 
In particular, if $\dim \sigma + \dim \tau < n-1$, we require that $\sigma$ and $w+\tau$ are disjoint.

\begin{definition}\label{defn:matroidDegen}
For a loopless matroid $\M$ of rank $r$ on $[n]$, let $\indeg_w\M$
be the union of those strata $V(\sigma) \subseteq X_{[n]}$
such that $\sigma$ intersects a maximal cone of $w+\Sigma_\M$, where $\sigma$ is a cone of $\Sigma_{[n]}$ of codimension $r-1$.
If $\mathrm{M}$ has loops, then define $\indeg_w\M$ to be empty. 
\end{definition} 

Note that $\indeg_w\M$ is by definition reduced.
The symbol $\indeg$ stands for ``initial degeneration''.
\footnote{We have avoided choosing the symbol ``$\operatorname{in}\M$'' to avert any confusion with the notion of the initial matroid of a valuated matroid, which we do not use.}

\begin{proposition}\label{prop:degenerindM}
Let $L \subseteq \kk^n$ be a realization of a loopless matroid $\mathrm{M}$. Then there is a flat degeneration of $W_L$ to $\indeg_w\M$ inside of $X_{[n]}$. In particular, $[\mathcal{O}_{W_L}] = [\mathcal{O}_{\indeg_w\M}]$ in $K_{\circ}(X_{[n]})$. 
\end{proposition}

\begin{proof}
Note that $\indeg_w\M$ does not change if we scale $w$, so we may assume that $w$ lies in $\mathbb{Z}^n/\mathbb{Z}\be_{[n]}$, the cocharacter lattice of the torus in $X_{[n]}$. In particular, $w$ defines an action of $\mathbb{G}_m$ on $X_{[n]}$. We claim that the flat limit $\lim_{t \to 0} t \cdot W_L$ is equal to $\operatorname{indeg}_w \M$. 

In \cite[Theorem 10.1]{Kat09}, Katz shows that, as $w$ is sufficiently generic, the limit is a union of torus-invariant strata. He also gives a combinatorial description of the strata which occur, and this description matches the definition of $\indeg_w\M$ because $\Sigma_\M$ is the tropicalization of $\PP L \cap \PP((\kk^*)^{[n]})$ \cite[\S9.3]{Stu02}.
In particular, this shows that the reduction of $\lim_{t \to 0} t \cdot W_L$ is equal to $\indeg_w\M$, so it suffices to show that this limit is reduced. 

By Example~\ref{ex:WL}, $W_L$ defines a kindred subscheme of $(\mathbb{P}^1)^{\binom{n}{2}}$. In particular, because Euler characteristics are locally constant in proper flat families, any flat degeneration of $W_L$ inside of $(\mathbb{P}^1)^{\binom{n}{2}}$ is kindred. Corollary~\ref{cor:CM} then implies that the flat limit is reduced. 

It remains to show that $[\mathcal{O}_{W_L}] = [\mathcal{O}_{\indeg_w\M}]$ in $K_{\circ}(X_{[n]})$. For any class $a \in K^{\circ}(X_{[n]})$, we have 
$$\chi(X_{[n]}, a \cdot [\mathcal{O}_{W_L}]) = \chi(W_L, a) = \chi(\indeg_w\M, a) = \chi(X_{[n]}, a \cdot [\mathcal{O}_{\indeg_w\M}])$$ because Euler characteristics are preserved by degeneration. By \cite[Theorem 1.3]{AP15}, the pairing $K^{\circ}(X_{[n]}) \times K_{\circ}(X_{[n]}) \to \mathbb{Z}$ given by $(a, b) \mapsto \chi(X_{[n]}, ab)$ is perfect, so $[\mathcal{O}_{W_L}] = [\mathcal{O}_{\indeg_w\M}]$. 
\end{proof}

\begin{remark}\label{lem:multiplicities indM}
Reducedness of the flat limit implies that all multiplicities in Katz' formula equal $1$.
That is, it is a consequence of the proof that for any codimension $r-1$ cone $\sigma$ of~$\Sigma_{[n]}$, there is at most one maximal cone of $w+\Sigma_\M$ which intersects it, 
and the tropical intersection multiplicity is~$1$.
This fact can also be proved using the agreement of stable intersection and matroid intersection in the sense of Welsh \cite[\S4]{Spe08} (see also \cite[Corollary 1.7]{EHL23}).
\end{remark}

Proposition~\ref{prop:degenerindM} implies that Euler characteristic computations on $\indeg_w\M$ are equal to the analogous  computations on the wonderful variety, and so to the analogous  computations in $K(\mathrm{M})$. In order to prove Theorem~\ref{thm:Macaulay}, we will need to relate computations on $\indeg_w\M$ to computations in $K(\mathrm{M})$ for non-realizable matroids. Our tool to do this will be \emph{valuativity}. 

For a set $P\subseteq\mathbb R^n$, let $\mathbf 1_P \colon \mathbb R^n\to\mathbb Z$ be its indicator function.
A function valued in an abelian group, with domain the set of all matroids of rank~$r$ on~$[n]$, 
is said to be \emph{valuative} if it factors through the map $\M\mapsto\mathbf 1_{P(\M)}$.
Valuative functions form an abelian group under addition.  

\begin{proposition}\label{prop:Kvaluative}
The assignment $\M \mapsto [\mathcal O_{\indeg_w\M}] \in K_\circ(X_{[n]})$ is valuative.
\end{proposition}

We prepare for the proof with the following lemma.

\begin{lemma}\label{lem:bigger cones}
A cone $\sigma$ of $\Sigma_{[n]}$, of any dimension, intersects $w+\Sigma_\M$
if and only if $V(\sigma)\subseteq\indeg_w\M$.
\end{lemma}

\begin{proof}
If $V(\sigma)$ is a component of $\indeg_w\M$, 
then the strata $V(\tau)$ contained in $V(\sigma)$ are those indexed by cones $\tau\supseteq\sigma$.
Because $\sigma$ contains a point $v$ of the relative interior of a face $F$ of $w+\Sigma_\M$ and $F$ meets $\sigma$ transversely,
a neighborhood of $v$ within~$F$ intersects every such cone $\tau$.

Conversely, suppose $\sigma$ is a cone of $\Sigma_{[n]}$ of codimension less than $r-1$ meeting $w+\Sigma_\M$.
By induction on the dimension, it suffices to find a proper face of $\sigma$ which meets $w+\Sigma_\M$.
By our transversality assumptions, the intersection $A$ of $w+\Sigma_\M$ with the subspace $\operatorname{span}_{\mathbb R}\sigma$ equals the support of their stable intersection.
Because $w+\Sigma_\M$ is a balanced polyhedral complex when every maximal cone is given weight~$1$,
the stable intersection $A$ also has positive weights on its facets making it balanced.
The dimension of $A$ is $\dim\Sigma_\M-\operatorname{codim}\sigma > 0$.
But $\sigma$ is a strictly convex cone, and a positive-dimensional balanced complex cannot be contained in  a strictly convex cone,
so $A$ meets the boundary of~$\sigma$.
\end{proof}

\begin{proof}[Proof of Proposition~\ref{prop:Kvaluative}]
Let $\M$ be a matroid of rank~$r$ on~$[n]$.
Let $\mathcal P_r$ denote the poset of all torus-orbit closures in $X_{[n]}$ of dimension at most $r-1$, ordered by inclusion.
Then $\mathcal P_r$ has the intersect-decompose property of \cite{Knu09} and contains the components of $\indeg_w\M$,
and we can use \cite[Theorem 1]{Knu09} to find integers $c_\sigma(\M)$ such that
\[
[\mathcal O_{\indeg_w\M}] = \sum_{V(\sigma) \in \mathcal P_r} c_\sigma(\M) [\mathcal O_{V(\sigma)}].
\]
The computation uses a recurrence, which we restate to be careful about the effect of using a larger intersect-decompose set than the minimal one.
First define 
\[i_\sigma(\M)=\begin{cases}
    1 & \text{if }V(\sigma)\subseteq\indeg_w\M,\\
    0 & \text{otherwise.}
\end{cases}\]
Now the $c_\sigma(\M)$ are defined by the following recursive property:
\[
c_\sigma(\M) = i_\sigma(\M) - \sum_{\substack{V(\sigma') \in \mathcal P_r\\ V(\sigma') \supsetneq V(\sigma)}} c_{\sigma'}(\M).
\]
The sum on the right side is empty when $\dim V(\sigma) = r-1$.
By descending induction on $\dim V(\sigma)$, if $i_\sigma(\M)=0$, then $c_\sigma(\M)=0$.
Then considering only those $\sigma$ for which $i_\sigma(\M)=1$ yields the M\"obius recurrence as Knutson states it.
 
By Lemma~\ref{lem:bigger cones},
$i_\sigma(\M)=1$ if and only if $\sigma\cap(w+\Sigma_\M)$ is nonempty.
If $\sigma\cap(w+\Sigma_\M)$ is nonempty, then it is the intersection of two tropically convex sets \cite[Proposition 2.14]{Ham15}, and therefore is tropically convex itself and hence contractible \cite[Theorem 2]{DS04}.
Because a polyhedral complex with a bounded face admits a deformation retract onto its bounded subcomplex (see, e.g., the proof of \cite[Lemma 3.4]{EMPV25}), for a sufficiently large cube $C$ centered at the origin, the intersection $\sigma\cap (w+ \Sigma_\M)$ deformation retracts to $C\cap \sigma\cap (w+\Sigma_\M)$.
In particular, we have $i_\sigma(\M) = \chi_{\text{top}}(C\cap \sigma\cap (w+\Sigma_\M))$, where $\chi_{\text{top}}$ is the topological Euler characteristic.
If $A$ and $B$ are bounded polyhedral sets, then one may find polyhedral complex structures for $A$, $B$, $A\cap B$, and $A\cup B$, compatible with inclusions, so $\chi_{\text{top}}(A) + \chi_{\text{top}}(B) = \chi_{\text{top}}(A\cup B) + \chi_{\text{top}}(A\cap B)$.
Because polyhedral sets are intersection closed, \cite{Gro78} then implies that the topological Euler characteristic of a bounded polyhedral set is linear in its indicator function.
Since the assignment $\M \mapsto \mathbf 1_{\Sigma_\M}$ is valuative \cite[Proposition 2.26 \& Remark 5.17]{Ham17}, \cite[Corollary 7.11(ii)]{BEST23}, the association $\M\mapsto i_\sigma(\M)$ is therefore valuative.
The recursive formula then implies that the coefficients $c_\sigma(\M)$ are also valuative.
\end{proof}

For any loopless matroid $\mathrm{M}$, there is a restriction map from $K^{\circ}(X_{[n]})$ to $K(\mathrm{M})$. In particular, for any class $\xi \in K^{\circ}(X_{[n]})$, we may consider the Euler characteristic $\chi(\mathrm{M}, \xi)$ of the image of $\xi$ in $K(\mathrm{M})$. 
When $\mathrm{M}$ is realized by $L$, this map coincides with the restriction map $K^{\circ}(X_{[n]}) \to K^{\circ}(W_L)$ followed by the identification $K^{\circ}(W_L) \cong K(\mathrm{M})$, see \cite[Section 5]{LLPP24}. 

\begin{proposition}\label{prop:Kclass}
For any class $\xi\in K^\circ(X_{[n]})$ and any loopless matroid $\mathrm{M}$, we have
\[\chi(\M,\xi)=\chi(\indeg_w\M,\xi).\]
\end{proposition}

\begin{proof}
When $\mathrm{M}$ is realizable, this follows from Proposition~\ref{prop:degenerindM}.  By Proposition~\ref{prop:Kvaluative}, the function which assigns a matroid $\mathrm{M}$ to $\chi(\indeg_w\M,\xi)$ is valuative. By \cite[Lemma 6.4]{LLPP24}, the function which assigns a matroid $\mathrm{M}$ to $\chi(\mathrm{M}, \xi)$ is valuative. 
If two valuative functions agree on all realizable matroids, then they are equal on all matroids \cite[Theorem 5.4]{DF10} (see also \cite[Lemma 5.9]{BEST23}).
\end{proof}

\begin{corollary}\label{cor:kindredinM}
For every matroid~$\M$, the embedding of $\indeg_w\M$ into $(\mathbb{P}^1)^{\binom{n}{2}}$ realizes it as a kindred subscheme. 
\end{corollary}

\begin{proof}
If $\M$ has loops, then $\indeg_w\M$ is empty, so this is automatic. We may therefore assume that $\M$ is loopless. 
Let $\mathcal{L}_{i,j}$ denote the restriction of the $\mathcal{O}(1)$ on the factor of $(\mathbb{P}^1)^{\binom{n}{2}}$ labeled by $\{i,j\}$. In \cite[Corollary 7.5]{LLPP24}, the authors show that
$$\chi(\M, \bigotimes_{i,j} \mathcal{L}_{i,j}^{\otimes a_{i,j}}) = \sum_{S} \prod_{\{i,j\} \in S} a_{i,j},$$
where the sum is over subsets $S$ of $\binom{[n]}2\coloneqq \{\{i,j\} : i, j \in [n],i\ne j\}$ that satisfy the \emph{dragon Hall--Rado condition}: for any nonempty $T \subset S$, it holds that $\rk_{\M}(\bigcup_{\{i,j\} \in T} \{i,j\}) \ge |T| + 1$. It is known that the subsets  which satisfy the dragon Hall--Rado condition are the independent sets of a matroid on~$\binom{[n]}2$ called the \emph{Dilworth truncation} of $\M$: see, e.g., 
\cite[\S1.1]{Mas81}. Proposition~\ref{prop:Kclass} implies that the Snapper polynomial of $\indeg_w\M$ has the same formula, and so $\indeg_w\M$ is a kindred subscheme, with the Dilworth truncation of~$\M$ as its progenitor matroid. 
\end{proof}

Using Corollary~\ref{cor:CM}, 
we deduce the following result from Corollary~\ref{cor:kindredinM}. 

\begin{corollary}\label{cor:connected}
$\indeg_w\M$ is Cohen--Macaulay, geometrically connected, and geometrically reduced.
\end{corollary}

We already have the tools to prove Theorems~\ref{thm:vanish} and \ref{thm:Macaulay} for ample line bundles on $X_{[n]}$, but we will need to develop an additional tool to prove the main theorems for nef line bundles. 
We start with the following fact, the so-called strong normality of generalized permutohedra.
Recall that, by our definitions, generalized permutohedra are lattice polytopes.

\begin{proposition}[{\cite[Theorem 18.6.3]{Wel76}}]\label{prop:strong normality}
Let $P_1,\ldots,P_\ell$ be a collection of generalized permutohedra in $\mathbb{R}^n$.
Every lattice point of the Minkowski sum $P_1+\cdots+P_\ell$ can be written as $p_1 + \dotsb + p_{\ell}$, with $p_i\in P_i\cap\mathbb Z^n$.
\end{proposition}

Recall that every nef line bundle on $X_{[n]}$ corresponds to a generalized permutohedron $P$. 
As noted after \cite[Proposition 3.10]{EHL23},
Proposition~\ref{prop:strong normality} implies that
the map $X_{[n]}\to\prod_{i=1}^\ell\mathbb{P}(\kk^{P_i\cap\mathbb{Z}^n})$ induced by the complete linear series of the line bundle associated to each~$P_i$
has image isomorphic to the toric variety of the normal fan of $P_1+\cdots+P_\ell$.
In particular, the image of the map $f_P \colon X_{[n]} \to \PP(\kk^{P \cap \ZZ^n})$ is the toric variety $X_P$ of the normal fan of $P$, and when $P_1,\ldots,P_{\binom n2}$ is the collection of all segments $\operatorname{conv}\{\be_i,\be_j\}$,
we recover the defining embedding $X_{[n]}\hookrightarrow(\mathbb P^1)^{\binom n2}$ of the introduction.

\begin{proposition}\label{prop:push}
Let $P$ be a generalized permutohedron, and let $f_P \colon X_{[n]} \to X_P$ be the corresponding map. Then $f_P$ is ``locally a coordinate projection from $(\mathbb{P}^1)^{\binom{n}{2}}$,'' in the following sense. 
There is an open cover of $X_P$ by open sets $\{U_{\sigma}\}$ such that each $U_{\sigma}$ admits an embedding into $(\mathbb{P}^1)^{\ell}$ for some $\ell$. Setting $Y_{\sigma}$ to be the closure of $U_{\sigma}$ in $(\mathbb{P}^1)^{\ell}$, we have the following commutative diagram
\begin{center}
\begin{tikzcd}
X_{[n]}  \arrow[d, "f_P"]
& f^{-1}_P(U_{\sigma}) \arrow[d] \arrow[l, hook'] \arrow[r, hook]
& X_{[n]} \arrow[r, hook] \arrow[d] &(\mathbb{P}^1)^{\binom{n}{2}} \arrow[d, "p"] \\
X_P 
& U_{\sigma} \arrow[l, hook'] \arrow[r, hook] & Y_{\sigma} \arrow[r, hook] & (\mathbb{P}^1)^{\ell},
\end{tikzcd}
\end{center}
where the map from $X_{[n]}$ to $(\mathbb{P}^1)^{\binom{n}{2}}$ is the defining embedding, and $p$ is a coordinate projection. 
\end{proposition}

\begin{proof}
Let $\sigma$ be a maximal cone of the normal fan of $P$, and let $U_{\sigma}$ be the corresponding open subset of $X_P$. Let $T_1, \dotsc, T_{\ell}$ be the directions of the edges leaving the vertex of $P$ corresponding to $\sigma$. Let $Q$ be the Minkowski sum of the simplices corresponding to $T_1, \dotsc, T_{\ell}$, where the simplex corresponding to an edge direction $\be_i - \be_j$ is $\operatorname{conv}\{\be_i, \be_j\}$. Note that $Q$ is a generalized permutohedron. Let $Y_{\sigma}$ be the toric variety corresponding to the normal fan of $Q$. Then $\sigma$ is also a cone of the fan of $Y_{\sigma}$, so $U_{\sigma}$ can be identified with an open subset of $Y_{\sigma}$. 
It follows from Proposition~\ref{prop:strong normality} that there is an embedding of $Y_{\sigma}$ into $(\mathbb{P}^1)^{\ell}$ induced by $T_1, \dotsc, T_{\ell}$. As the embedding of $X_{[n]}$ into $(\mathbb{P}^1)^{\binom{n}{2}}$ is the map induced by taking all $\binom{n}{2}$ of the simplices $\operatorname{conv}\{\be_i, \be_j\}$, this gives the desired diagram. 
\end{proof}

\begin{corollary}\label{prop:CM}
Let $P$ be a generalized permutohedron, and let $f_P \colon X_{[n]} \to X_P$ be the corresponding map. Then, for any matroid $\M$, the natural map $\mathcal{O}_{f_P(\indeg_w\M)} \to Rf_{P*} \mathcal{O}_{\indeg_w\M} $ is an isomorphism, and $f_P(\indeg_w\M)$ is Cohen--Macaulay. If $\M$ is realized by $L$, then the natural map $\mathcal{O}_{f_P(W_L)}\to Rf_{P*} \mathcal{O}_{W_L}$ is an isomorphism, and $f_P(W_L)$ is Cohen--Macaulay. 
\end{corollary}

\begin{proof}
Let $U_{\sigma}$ be an element of the open cover constructed in Proposition~\ref{prop:push}, which is equipped with a locally closed embedding $U_{\sigma} \hookrightarrow (\mathbb{P}^1)^{\ell}$. By Corollary~\ref{cor:kindredinM}, $\indeg_w\M$ is a kindred subscheme of $(\mathbb{P}^1)^{\binom{n}{2}}$. By Proposition~\ref{prop:projection}, $p(\operatorname{\indeg_w\M})$ is a kindred subscheme of $(\mathbb{P}^1)^{\ell}$, and the natural map from $\mathcal{O}_{p(\indeg_w\M)}$ to $Rp_* \mathcal{O}_{\indeg_w\M}$ is an isomorphism. By Corollary~\ref{cor:CM}, $p(\indeg_w\M)$ is Cohen--Macaulay, so $U_{\sigma} \cap p(\indeg_w\M)$ is Cohen--Macaulay. As being Cohen--Macaulay is a local property and $f_P(\indeg_w\M)$ is covered by open sets of the form $U_{\sigma} \cap p(\indeg_w\M)$, this implies that $f_P(\indeg_w\M)$ is Cohen--Macaulay. Similarly, that the natural map $\mathcal{O}_{f_P(\indeg_w\M)}\to Rf_{P*} \mathcal{O}_{\indeg_w\M}$ is an isomorphism is a local property, so this holds as well. When $\M$ is realized by $L$, since $W_L$ is kindred (Example~\ref{ex:WL}), an identical argument shows that the natural map $\mathcal{O}_{f_P(W_L)}\to Rf_{P*} \mathcal{O}_{W_L}$ is an isomorphism and that $f_P(W_L)$ is Cohen--Macaulay.
\end{proof}

\begin{proposition}\label{prop:vanishinM}
Let $P$ be a generalized permutohedron, and let $f_P \colon X_{[n]} \to X_P$ be the corresponding map. Then $H^i(\indeg_w\M, \mathcal{L}_P^{\otimes a}) = 0$ unless $a \ge 0$ and $i = 0$, or $a < 0$ and $i = \dim f_P(\indeg_w\M)$. For any $a \ge 0$, the restriction map from $H^0(X_{[n]}, \mathcal{L}_P^{\otimes a}) \to H^0(\indeg_w\M, \mathcal{L}_P^{\otimes a})$ is surjective.
\end{proposition}

\begin{proof}
The toric variety $X_P$ is equipped with an ample line bundle $\mathcal{O}(1)$ corresponding to the polytope $P$; this line bundle pulls back to $\mathcal{L}_P$ under $f_P$. 
By Corollary~\ref{prop:CM} and the projection formula, the natural map from $H^i(f_P(\indeg_w\M), \mathcal{O}(a))$ to $H^i(\indeg_w\M, \mathcal{L}_P^{\otimes a})$ is an isomorphism, so it suffices to show the desired vanishing and surjectivity of global sections for $H^i(f_P(\indeg_w\M), \mathcal{O}(a))$. 

Note that $f_P(\indeg_w\M)$ is a union of torus-invariant subvarieties in the toric variety $X_P$. As such, it is defined over $\operatorname{Spec} \mathbb{Z}$. By upper semicontinuity, it suffices to show the desired vanishing and surjectivity of global sections after base changing $f_P(\indeg_w\M)$ to a field of positive characteristic. Over any field of positive characteristic, $f_P(\indeg_w\M)$ is a compatibly split subscheme of $X_P$ with respect to a certain Frobenius splitting of $X_P$ \cite[Proposition 3.2]{Pay09}; see also \cite[Proposition 1.2.1]{BK05}. 
The vanishing and surjectivity of global sections then follow from \ref{item:1} and \ref{item:2} in Lemma~\ref{lem:splittingprop}, as $f_P(\indeg_w\M)$ is Cohen--Macaulay by Corollary~\ref{prop:CM}. 
\end{proof}

\begin{proof}[Proof of Theorem~\ref{thm:vanish}]
Note that \ref{item:push} was proved in Corollary~\ref{prop:CM}. 
For parts \ref{item:vanish} and \ref{item:surject}, we use the degeneration of $W_L$ to $\indeg_w\M$ constructed in Proposition~\ref{prop:degenerindM}; these then follow from upper semicontinuity and Proposition~\ref{prop:vanishinM}. 
\end{proof}

\begin{proof}[Proof of Theorem~\ref{thm:Macaulay}]
By Proposition~\ref{prop:Kclass}, $\chi(\M, \mathcal{L}_P^{\otimes a}) = \chi(\indeg_w\M, \mathcal{L}_P^{\otimes a})$.
As $H^i(\indeg_w\M, \mathcal{L}_P^{\otimes a}) = 0$ for any $a \ge 0$ and $i > 0$, we see that 
$\chi(\indeg_w\M, \mathcal{L}_P^{\otimes a}) = \dim H^0(\indeg_w\M, \mathcal{L}_P^{\otimes a})$ 
and in particular, the right-hand side is a polynomial in $a$ of some degree $d$.
Using Proposition~\ref{prop:strong normality} and the results in Proposition~\ref{prop:vanishinM}, we see that the ring $\bigoplus_{a \ge 0} H^0(\indeg_w\M, \mathcal{L}_P^{\otimes a})$ is Cohen--Macaulay and generated in degree~$1$. 
As described in the introduction, this implies that if we write
$$\sum_{a \ge 0}\dim H^0(\indeg_w\M, \mathcal{L}_P^{\otimes a})t^a = \frac{h_0^* + h_1^*t +  \dotsb + h_d^*t^d}{(1 - t)^{d+1}},$$
then $(h_0^*, \dotsc, h_d^*)$ is a Macaulay vector. The result follows. 
\end{proof}

\section{Dimensions of images}\label{sec:dimensions}

Theorem~\ref{thm:vanish} states that the cohomology of a negative tensor power of a line bundle $\mathcal{L}_P$ on a wonderful variety vanishes except possibly in degree equal to the dimension of $f_P(W_L)$, where $f_P$ is the map to projective space induced by $\mathcal{L}_P$. 
In Theorem~\ref{thm:numericaldim} below, we give an explicit combinatorial formula for that dimension.
Theorem~\ref{thm:numericaldim} is stated in terms of $\indeg_w\M$, but if $L$ realizes $\M$, then $W_L$ degenerates to $\indeg_w\M$ by Proposition~\ref{prop:degenerindM}, and so they have the same class in the Chow groups of $X_{[n]}$. This implies that $\dim f_P(W_L) = \dim f_P(\indeg_w\M)$ for any generalized permutohedron $P$ since, if $A$ is an ample divisor on $X_{[n]}$, then $\dim W_L - \dim f_P(W_L)$ is the least $i$ such that $f_{P*}(A^i \smallfrown [W_L])$ is nonzero, and similarly for $\indeg_w\M$ because it is equidimensional. 

\smallskip
For a loopless matroid $\operatorname{M}$, we begin by showing that $\dim f_P(\indeg_w\M)$ is equal to the degree of the Snapper polynomial $\chi(\mathrm{M}, \mathcal{L}_P^{\otimes a})$, which is used to define the $h^*$ vector, and is also equal to a quantity derived from the \emph{Chow ring} $A^\bullet(\M)$ of $\M$. This will be used to deduce Theorem~\ref{thm:numericaldim}. For a rank $r$ loopless matroid $\M$, let $\int_\M: A^{r-1}(\M) \to \ZZ$ be the degree map as defined in \cite[Section 5]{AHK18}.

\begin{proposition}\label{prop:diminM}
Let $\operatorname{M}$ be a loopless matroid on $[n]$, and let $P$ be a generalized permutohedron in $\mathbb{R}^n$. Then the following quantities are equal:
\begin{enumerate}
\item The degree of the polynomial $a \mapsto \chi(\operatorname{M}, \mathcal{L}_P^{\otimes a})$.
\item The dimension of $f_P(\indeg_w\M)$. 
\item\label{item:numdim} The largest $d$ such that $c_1(\mathcal{L}_P)^d$ is nonzero in $A^\bullet(\operatorname{M})$, i.e., the {numerical dimension} of $c_1(\mathcal{L}_P)$ in $A^\bullet(\operatorname{M})$. 
\end{enumerate}
\end{proposition}

\begin{proof}
Note that the line bundle $\mathcal{O}(1)$ on $\PP(\kk^{P \cap \ZZ^n})$ pulls back to $\mathcal{L}_P$ via $f_P$. 
By Proposition~\ref{prop:Kclass} and Corollary~\ref{prop:CM}, we have 
$\chi(\operatorname{M}, \mathcal{L}_P^{\otimes a}) = \chi(\indeg_w\M, \mathcal{L}_P^{\otimes a}) = \chi(f_P(\indeg_w\M), \mathcal{O}(a)).$
The line bundle $\mathcal{O}(1)$ is ample on $f_P(\indeg_w\M)$, so its Snapper polynomial has degree equal to the dimension of $f_P(\indeg_w\M)$, proving that the first quantity is equal to the second quantity. 

For the third quantity, we first note from \cite[Section 5]{AHK18} that
$A^\bullet(\M)$ is equipped with a surjection $A^\bullet(X_{[n]}) \to A^\bullet(\M)$, such that, for any $\xi\in A^{r-1}(X_{[n]})$, we have $\int_\M(\xi) =  \xi \smallfrown [\Sigma_\M] $. Here $[\Sigma_\M]\in A_{r-1}(X_{[n]})$ is the Chow homology class called the \emph{Bergman class} of $\M$, which is the Minkowski weight corresponding to $\Sigma_\M$ as a balanced fan.

Let $d$ be the numerical dimension of $c_1(\mathcal L_P)$.
For $0\leq k \leq d$, Poincar\'{e} duality for $A(\operatorname{M})$ \cite[Theorem 6.19]{AHK18} shows that there is some $x_k\in A^{r-1-k}(X_{[n]})$ such that $\int_{\M} x_k \cdot c_1(\mathcal{L}_P)^{k} \neq 0$.  As the ample cone is full dimensional and $A^{\bullet}(\M)$ is generated in degree $1$, we may take $x_k = A^{r-1-k}$ for some ample $A \in A^1(X_{[n]})$ and each $0\leq k \leq d$.
Because ${f_P}_*([\indeg_w\M] \cdot A^{r-1-k})$ is effective and $\mathcal O(1)$ is ample, we see that ${f_P}_*(A^{r-1-k} \smallfrown [\indeg_w\M]  )$ is nonzero if and only if $c_1(\mathcal O(1))^k \smallfrown {f_P}_*(A^{r-1-k} \smallfrown [\indeg_w\M]  ) $ is nonzero.
The latter is equal to $A^{r-1-k} \cdot c_1(\mathcal L_P)^k \smallfrown [\indeg_w \M] $ by the projection formula, which in turn is equal to $\int_\M  A^{r-1-k} \cdot c_1(\mathcal L_P)^k$ because the fundamental class of $\indeg_w\M$ in $A_{\bullet}(X_{[n]})$ is equal to the Bergman class $[\Sigma_{\M}]$ by Proposition~\ref{prop:Kclass} and the Hirzebruch--Riemann--Roch theorem, used like it is in Proposition~\ref{prop:HRR}. 
Therefore $r-1-d$ is the smallest number such that ${f_P}_*(A^{r-1-d} \smallfrown [\indeg_w\M]  )$ is nonzero, and so $r-1-d = \dim \indeg_w \M - \dim f_P(\indeg_w \M)$ since $\indeg_w\M$ is equidimensional of dimension $r-1$. Rearranging implies the result.
\end{proof}

We can now give a combinatorial formula for the quantities in Proposition~\ref{prop:diminM}. This formula will imply that the numerical dimension of $c_1(\mathcal{L}_P)$ in $A^\bullet(\M)$ only depends on the lineality space of the normal fan of $P$. 
A generalized permutohedron $P$ in $\mathbb{R}^n$ induces a partition $[n] = S_1 \sqcup \dotsb \sqcup S_{\ell}$, where the sets are the equivalence classes induced by the equivalence relation generated by setting $i \sim j$ if there is an edge of direction $\be_i - \be_j$ in $P$.
The parts of the partition induced by a matroid polytope $P(\M)$ are usually called the \emph{connected components} of~$\M$.

\begin{theorem}\label{thm:numericaldim}
Let $\M$ be a loopless matroid on $[n]$, and let $P$ be a generalized permutohedron with induced partition $[n] = S_1 \sqcup \dotsb \sqcup S_{\ell}$. Then the numerical dimension of $c_1(\mathcal{L}_P)$ in $A^\bullet(\mathrm{M})$ is the minimum over partitions $[n] = T_1 \sqcup \dotsb \sqcup T_k$ coarsening $[n] = S_1 \sqcup \dotsb \sqcup S_{\ell}$ of $\sum_{i=1}^{k} (\rk_{\mathrm{M}}(T_i) - 1)$.
\end{theorem}

We will need a preparatory lemma. For a nonempty set $S\subseteq[n]$, write $\Delta_S$ for the convex hull of $\{\be_i:i\in S\}$, which is a generalized permutohedron, and let $h_S = c_1(\mathcal{L}_{\Delta_S})$.

\begin{lemma}\label{lem:simplicialcone}
Let $[n] = S_1 \sqcup \dotsb \sqcup S_{\ell}$ be a partition, and let $Q = \sum_{i = 1}^{\ell} \sum_{\emptyset \not= S \subseteq S_i} \Delta_{S}$. 
Let $d$ be the minimum over partitions $[n] = T_1 \sqcup \dotsb \sqcup T_k$ coarsening $[n] = S_1 \sqcup \dotsb \sqcup S_{\ell}$ of $\sum_{i=1}^{k} (\rk_{\mathrm{M}}(T_i) - 1)$.  Then $\int_{\M} c_1(\mathcal{L}_Q)^d h_{[n]}^{r - 1 - d} > 0$, and the numerical dimension of $c_1(\mathcal{L}_Q)$ is $d$. 
\end{lemma}

\begin{proposition}\label{prop:DHR} \cite[Theorem 5.2.4]{BES}
For a multiset $\{J_1, \dotsc, J_{r-1}\}$ of subsets of $[n]$ and a loopless matroid $\M$ on $[n]$, $\int_{\M} h_{J_1} \dotsb h_{J_{r-1}}$ is equal to $1$ if $\{J_1, \dotsc, J_{r-1}\}$ satisfies the \emph{dragon Hall--Rado condition}:
\[
\rk_\M\left(\bigcup_{i \in I} J_i\right) \geq |I| + 1 \quad\text{for all nonempty } I \subseteq [r-1],
\]
and is $0$ otherwise. 
\end{proposition}

\begin{proof}[Proof of Lemma~\ref{lem:simplicialcone}]
To prove that $\int_{\M} c_1(\mathcal{L}_Q)^d h_{[n]}^{r - 1 - d} > 0$, it suffices to check that there is some multiset $\{S_{i_1}, \dotsc, S_{i_d}\}$ that satisfies the dragon Hall--Rado condition after adding $r - 1 - d$ copies of $[n]$. 
Let $a_i$ be the multiplicity of~$S_i$ in such a family.
This holds for $(a_i)\in\mathbb N^\ell$ if and only if
\[\sum_{i\in I}a_i\le\rk_\M\left(\bigcup_{i\in I}S_i\right)-1\]
for every nonempty $I\subseteq[\ell]$.
By \cite[Theorem 8]{Edm70} the nonnegative integer vectors $(a_i)$ satisfying these bounds are the lattice points of a polymatroid in Edmonds' sense.
The maximum attained by $\sum_{i=1}^\ell a_i$ is given by the evaluation of its rank function at~$[\ell]$,
which, from the description in \cite[Theorem 8]{Edm70}, is exactly $d$. 
It follows from \cite[Corollary 7.5]{LLPP24}, which gives an explicit formula for the Snapper polynomial $\chi(\M, \bigotimes_S \mathcal L_{\Delta_S}^{\otimes a_S})$ in terms of the dragon Hall--Rado condition, 
that the degree of the polynomial $a \mapsto \chi(\M, \mathcal{L}_Q^{\otimes a})$ is $d$, and so the numerical dimension of $c_1(\mathcal{L}_Q)$ is $d$ by Proposition~\ref{prop:diminM}. 
\end{proof}

As the last item of preparation, we recall a straightforward consequence of \cite[Proposition 4.12(2)]{EL}.

\begin{proposition}\label{prop:monotone}
Let $D_1$ and $D_2$ be the restrictions of nef classes from $A^1(X_{[n]})$ to $A^1(\M)$, where $\M$ has rank~$r$. 
Let $E$ be the restriction of an effective class in $A^1(X_{[n]})$, and assume that $D_2 + E$ is the restriction of a nef class from $A^1(X_{[n]})$. Then, for each $0 \le a \le r-1$,
$$\int_{\mathrm{M}}D_1^a D_2^{r-1-a} \le \int_{\mathrm{M}}D^a_1 (D_2 + E)^{r - 1 - a}.$$
\end{proposition}

\begin{proof}
We find
$$
\int_{\mathrm{M}} \big( D^a_1 (D_2 + E)^{r - 1 - a} - D_1^a D_2^{r-1-a} \big)\\
= \int_{\mathrm{M}} D_1^a E \big((D_2+E)^{r-1-a-1} + (D_2+E)^{r-1-a-2}D_2 + \dotsb + D_2^{r-1-a-1}\big),
$$
which is nonnegative by \cite[Proposition 4.12(2)]{EL}, which is a special case of \cite[Theorem 8.9]{AHK18}.
\end{proof}

\begin{proof}[Proof of Theorem~\ref{thm:numericaldim}]
Let $Q$ be the Minkowski sum $\sum_{i = 1}^\ell \sum_{\emptyset \subsetneq S \subseteq S_i} \Delta_S$, and let $d$ and $e$ be the numerical dimensions of  $c_1(\mathcal L_P)$ and $c_1(\mathcal L_Q)$, respectively.
By Lemma~\ref{lem:simplicialcone}, it suffices to show that $d = e$.

As $Q$ is the product of permutohedra, one from each $\RR^{S_i}$, the image of $X_{[n]}$ under $f_Q$ is $X_{S_1} \times \dotsb \times X_{S_{\ell}}$.
The partition $[n] = S_1 \sqcup \dotsb \sqcup S_\ell$ has the property that $P$ is a product $P_1 \times \dotsb \times P_\ell$ of generalized permutohedra $P_i \subset \RR^{S_i}$, each of maximal dimension $|S_i| - 1$ \cite[Section 5.1]{AA23}.
Thus, the map $f_P$ factors through $f_Q$, and so $\dim f_P(\indeg_w\M) \le \dim f_Q(\indeg_w\M)$, i.e., $d\leq e$.

The generalized permutohedron $P$ defines a nef and big line bundle on $X_{S_1} \times \dotsb \times X_{S_{\ell}}$. In particular, by \cite[Corollary 2.2.7]{Laz04a}, we can write $mc_1(\mathcal{L}_P) = c_1(\mathcal{L}_Q) + E$ for some $m > 0$, where  $E$ is the pullback of an effective divisor from $X_{S_1} \times \dotsb \times X_{S_\ell}$ (and so is effective). By Lemma~\ref{lem:simplicialcone}, $\int_{\M} c_1(\mathcal{L}_Q)^{e} h_{[n]}^{r - 1 - e} > 0$. By Proposition~\ref{prop:monotone}, we have
\[0< \int_{\mathrm{M}}  c_1(\mathcal{L}_Q)^e h_{[n]}^{r-1-e} \le \int_{\mathrm{M}}  (c_1(\mathcal{L}_Q)+ E)^{e} h_{[n]}^{r-1-e} = m^e\int_{\mathrm{M}} c_1(\mathcal{L}_P)^{e} h_{[n]}^{r-1-e}.\]
In particular, $c_1(\mathcal{L}_P)^{e}\neq 0$, so $e\leq d$.
\end{proof}

We record a case in which the formula of Theorem~\ref{thm:numericaldim} is particularly easy to evaluate.

\begin{corollary}\label{cor:numericaldim coarsening}
If $\M$ is a loopless matroid of rank~$r$ on~$[n]$ and $P$ a generalized permutohedron such that 
the induced partition $S_1\sqcup\cdots\sqcup S_\ell$ of~$P$ coarsens the induced partition of~$P(\M)$ (into connected components),
then the numerical dimension of $c_1(\mathcal{L}_P)$ equals $r-\ell = r-n+\dim P$. 
\end{corollary}

\begin{proof}
The rank function $\rk_\M$ is additive on unions of connected components \cite[Fact 4.2.13]{Oxl11}.
Any partition $T_1\sqcup\cdots\sqcup T_k$ coarsening the induced partition of~$P$ also coarsens the partition into connected components of~$\M$, and therefore
\[\sum_{i=1}^k(\rk_\M(T_i)-1) = \rk_\M([n])-k = r-k\]
which is minimized by maximizing~$k$, i.e., taking $T_\bullet$ and $S_\bullet$ identical with $k=\ell$.
\end{proof}

While the Snapper polynomial of $\mathcal L_P$ on a matroid is valuative, we caution that the $h^*$ vector is not valuative because of the varying powers of $1-t$ that depend on the dimension, as computed in this section. Explicitly, let $\M$ be the rank 2 uniform matroid on $[4]$ (i.e., with all sets of size~2 as bases), let $\M_1$ and $\M_2$ be rank~2 matroids on~$[4]$ with unique non-basis $\{1,2\}$ and $\{3,4\}$, respectively, and let $\M_{12}$ be the rank~2 matroid on~$[4]$ with both of these as non-bases.  One can verify that with $P = P(\M_{12})$, the $h^*$ vectors do not respect the valuative relation $\mathbf 1_{P(\M)} = \mathbf 1_{P(\M_1)} + \mathbf 1_{P(\M_2)} - \mathbf 1_{P(\M_{12})}$.

\section{Discussion and examples}\label{sec:discussion}

In this section, we give some examples of nef divisors which do not satisfy the conclusion of Theorem~\ref{thm:vanish} and Theorem~\ref{thm:Macaulay},
and we discuss a few extensions of the main theorems.

\subsection{Examples}\label{sec:examples}

In the Hodge theory of matroids \cite{AHK18} and more generally in tropical Hodge theory \cite{ADH23}, a positivity property known as the ``K\"ahler package'' is established for the set of \emph{strictly convex divisor classes} \cite[Definition 4.1 \& 4.2]{AHK18} (see also \cite[Definition 5.1]{ADH23}), which contains, in general strictly, the pullbacks of ample divisors on the permutohedral variety.
In contrast, the following example shows that the conclusion of Theorem~\ref{thm:Macaulay} may fail for the line bundle of a ``strictly convex'' divisor class. See also \cite[Example 5.18]{Cheng}.

\begin{example}\label{ex:grid}
For $n\geq 2$, let $[n,\overline n] := \{1, \dotsc, n, \overline 1, \dotsc, \overline n\}$, and define $L = \CC^3 \hookrightarrow \CC^{[n, \overline n]}$ by $(x,y,z) \mapsto (x-z, x-2z, \dotsc, x-nz, y-z, y-2z, \dotsc, y-nz)$.  The corresponding hyperplane arrangement on $\PP L = \PP^2$ consists of two families of parallel lines $\{\ell_i := V(x-iz)\}_{i\in [n]}$ and $\{\ell_{\overline j} := V(y-jz)\}_{\overline j\in [\overline n]}$.
The wonderful variety $W_L$ is the blow-up of $n^2+2$ points consisting of $[0:1:0] = \bigcap_{i\in [n]}\ell_i$, $[1:0:0] = \bigcap_{\overline j\in [\overline n]} \ell_{\overline j}$, and $\ell_i \cap \ell_{\overline j}$ for every pair $(i,\overline j)$.
Let $x_{[n]}$, $x_{[\overline n]}$, and $x_{i,\overline j}$ respectively denote the exceptional divisors, and let $x_i$ be the strict transform of $\ell_i$ for $i\in [n,\overline n]$.
Denote by $\partial W_L$ the set of divisors $\{x_{i,\overline j}\}_{i\in [n], \overline j \in [\overline n]}\cup \{x_{[n]}, x_{[\overline n]}\} \cup \{x_i\}_{i\in [n,\overline n]}$.
For a divisor $D$ on $W_L$, one can translate \cite[Definition 5.1]{ADH23} to the statement that $D$ is ``strictly convex'' if and only if $D$ is $\RR$-linearly equivalent to a positive linear combination of $\partial W_L$ and the intersection product $D.E$ is positive for all $E \in \partial W_L$. 

Let $\alpha$ denote the pullback of the hyperplane class along the blow-down $W_L \to \PP^2$, and define a divisor class $D\in A^1(W_L)$ by
\[
D := (n+2)\alpha - x_{[n]} - x_{[\overline n]} - \sum_{i\in [n],\overline j \in [\overline n]} x_{i,\overline j}.
\]
Since $\alpha = x_{\overline{k}} + \sum_{i\in [n]} x_{i,\overline k} + x_{[\overline n]} = x_k + \sum_{i\in [n]} x_{k,\overline i} + x_{[n]}$ in $A^1(W_L)$ for any fixed $k\in [n]$, one verifies that $D$ is $\RR$-linearly equivalent to a positive linear combination of the divisors in $\partial W_L$.  Moreover, one verifies that $D.E = 1$ for all $E\in \partial W_L$, so 
$D$ is ``strictly convex.''
By Riemann--Roch for surfaces, one computes
\[
\chi(W_L, \mathcal O_{W_L}(aD)) = \frac{D.D}{2}a^2 - \frac{D.K_{W_L}}{2}a + \chi(W_L, \mathcal O_{W_L}) = (2n+1)a^2 - \frac{(n+1)(n-4)}{2}a +1.
\]
From the relation $\chi(W_L, \mathcal O_{W_L}(aD)) = h_0^*\binom{a+2}{2} + h_1^*\binom{a+1}{2} + h_2^*\binom{a}{2}$, we find that
\[
(h_0^*, h_1^*, h_2^*) = \left( 1, \frac{-n^2+7n+2}{2},\frac{(n+1)n}{2} \right)\]
in this case.
In particular, the $h^*$ vector is not a Macaulay vector for $n\geq 6$: $n=6$ gives $(1,4,21)$, $n=7$ gives $(1,1,28)$, and $n\geq 8$ gives $h_1^*<0$.
One can verify directly using submodular functions that the divisor class $D$ is indeed not a pullback of a nef divisor on the permutohedral variety when $n>3$.
\end{example}

The next example shows that the conclusion of Theorem~\ref{thm:Macaulay} may fail for a line bundle even if there is a tensor power that is isomorphic to $\mathcal L_P$ for some generalized permutohedron $P$.

\begin{example}\label{ex:fano}
Let $\kk$ be a field of characteristic $2$, and let $L \subseteq \kk^7$ be a $3$-dimensional subspace whose intersections with the coordinate hyperplanes of~$\kk^7$ are, for some choice of coordinates on~$L$, the $7$ planes defined over~$\mathbb{F}_2$. The matroid realized by~$L$ is known as the \emph{Fano matroid}. 
Then $W_L$ is obtained by blowing up the $7$ $\mathbb{F}_2$-rational points of $\mathbb{P}^2$. Let $D$ be the anti-canonical divisor of $W_L$, so $D = \mathcal{O}(3) \otimes \mathcal{O}(-E_1 - \dotsb - E_7)$, where the $E_i$ are the exceptional divisors. The line bundle $\mathcal{O}(2D)$ is the restriction of the line bundle $\mathcal{L}_{-\Delta_{[7]}}$ on~$X_{[7]}$ corresponding to the divisor class usually denoted by $\beta$. 

However, $\mathcal{O}(D)$ is not the restriction of a nef line bundle from $X_{[7]}$. 
That is, there is no integer-valued submodular function $\rk \colon 2^{[7]}\to\mathbb Z$ taking value $0$ at the empty set and singletons, $-1$ at the three-element flats of the Fano matroid (the triples of coordinates of $\mathbb P(\kk^7)$ vanishing at $\mathbb F_2$-rational points of~$\mathbb PL$), and $-3$ at~$[7]$. We sketch a verification.
First, $\rk$ takes value at most $-2$ on sets of size~$5$, because these are unions of two three-element flats.
The assumptions imply $\rk$ is nonincreasing, so $\rk(S)\in\{0,-1\}$ when $|S|=2$. 
If $\rk(S)=-1$, submodularity forces that $\rk([7]\setminus S)=-2$ and, for any $T\subseteq[7]$, that $\rk(T)=\rk(T\cap S)+\rk(T\setminus S)$.
When $T$ is a three-element flat meeting $S$ in one element, this shows $\rk(T\setminus S)=-1$. Applying submodularity to two of these differences and one three-element flat disjoint from~$S$ gives the contradiction $\rk([7]\setminus S)\le-3$.
Otherwise $\rk(S)=0$ whenever $|S|=2$.
Then applying submodularity to a set of size~$5$ and two three-element flats that meet it in two elements implies the contradiction $\rk([7])\le-4$.

Let $x, y, z$ be coordinates on $\mathbb{P}^2$. The dimension of the space of sections $H^0(W_L, \mathcal{O}(D))$ is $3$, and the sections $a = xy(x+y), b = xz(x+z), c = yz(y+z)$ form a basis for $H^0(W_L, \mathcal{O}(D))$. The sections define a map $\pi \colon W_L \to \mathbb{P}^2$, which is generically $2$ to $1$. Let $d$ be the section $xyz(x+y)(x+z)(y+z)$ of $\mathcal{O}(2D)$. Then $H^0(W_L, \mathcal{O}(2D))$ is generated by $\operatorname{Sym}^2 H^0(W_L, \mathcal{O}(D))$ and $d$. The complete linear system of $H^0(W_L, \mathcal{O}(2D))$ defines a map $f$ from $W_L$ to $\mathbb{P}^6$, whose image is the variety usually called the \emph{reciprocal plane} \cite{PS06}, and this map factors through $\mathbb{P}(1,1,1,2)$. This map contracts the strict transforms of the $\mathbb{F}_2$-rational lines on $\mathbb{P}^2$ to obtain $7$ $A_1$ surface singularities. As $A_1$ surface singularities are rational 
and Gorenstein, we see that $Rf_* \mathcal{O}_{W_L} \simeq \mathcal{O}_{f(W_L)}$. 

Viewing $f(W_L)$ as a hypersurface in $\mathbb{P}(1,1,1,2)$, we see that the section ring is $\kk[a,b,c,d]/(d^2 - abc(a+b+c))$, which is Gorenstein. This describes $f(W_L)$ as a double cover of $\mathbb{P}^2$, branched along the quartic $abc(a+b+c) = 0$. We see that $\pi_* \mathcal{O}_{W_L} \simeq \mathcal{O}_{\mathbb{P}^2} \oplus \mathcal{O}_{\mathbb{P}^2}(-2)$ and $R^i \pi_* \mathcal{O}_{W_L} = 0$ for $i > 0$, and that the $h^*$ vector of $\mathcal{O}(D)$ is $(1,0,1)$. This is not a Macaulay vector.
\end{example}

\smallskip
The following example shows that  Theorem~\ref{thm:vanish}\ref{item:vanish} does not extend to restrictions to $W_L$ of globally generated vector bundles on $X_{[n]}$. 

\begin{example}
In \cite{BEST23}, the authors introduced a vector bundle $\mathcal{Q}_L$ on~$X_{[n]}$ associated to a linear subspace $L$ of~$\kk^n$.
This bundle $\mathcal{Q}_L$ is globally generated, and it is equipped with a section that transversely cuts out $W_L$ \cite[Theorem 7.10]{BEST23}, and so the restriction of $\mathcal{Q}_L$ to $W_L$ is identified with the normal bundle of $W_L$ in $X_{[n]}$. As such, $H^0(W_L, \mathcal{Q}_L)$ is identified with the tangent space to $[\mathcal{O}_{W_L}]$ in the Hilbert scheme of $X_{[n]}$, and $H^1(W_L, \mathcal{Q}_L)$ is the obstruction space to deforming $W_L$ inside of $X_{[n]}$ \cite[Proposition 6.5.2]{FGAExplained}.
Let $\kk = \mathbb{C}$, and let $\M$ be the matroid of rank $3$ on $[12]$ appearing in \cite[Theorem 4.7]{CoreyLuber}. This matroid has the property that $\operatorname{Gr}_{\M}$, the locally closed subset of $\operatorname{Gr}(3, 12)$ which consists of realizations of $\M$, is a $12$-dimensional scheme whose singular locus consists of two torus-orbits. The tangent space to a point in either of these orbits has dimension $13$. One can show that if $L$ is a realization of $\M$, then an open neighborhood of $[\mathcal{O}_{W_L}]$ in the Hilbert scheme of $X_{[12]}$ is isomorphic to an open neighborhood of $[L]$ in~$\operatorname{Gr}_{\M}$; see \cite[Theorem 7.6]{KPosNotes}. In particular, if $L$ corresponds to one of the smooth points of $\operatorname{Gr}_{\mathrm{M}}$, then $\dim H^0(W_L, \mathcal{Q}_L) = 12$, but if $L$ corresponds to one of the singular points of $\operatorname{Gr}_{\M}$, then $\dim H^0(W_L, \mathcal{Q}_L) = 13$ and $\dim H^1(W_L, \mathcal{Q}_L) > 0$. 
\end{example}

\subsection{Vanishing theorems for other varieties} 
We now explain how Theorem~\ref{thm:vanish} can be used to deduce vanishing theorems for some variations on wonderful varieties. Given a loopless matroid $\mathrm{M}$, a \emph{building set} $\mathcal{G}$ is a subset of the lattice of flats of $\mathrm{M}$ which satisfies some combinatorial conditions, given in \cite[Definition 4.4]{FK}. Recall that $\Delta_S$ denotes the convex hull of $\{\be_i : i \in S\}$.  Set
$$P_{\mathcal{G}} = \sum_{S \colon \operatorname{cl}_{\mathrm{M}}(S) \in \mathcal{G}} \Delta_S,$$
where the sum denotes Minkowski sum and $\operatorname{cl}_{\M}$ sends a set $S$ to the smallest flat containing it. Note that $P_{\mathcal{G}}$ is a generalized permutohedron. The \emph{wonderful variety of }$\mathcal{G}$, denoted $W_{L, \mathcal{G}}$, is the image of $W_L$ under the map from $X_{[n]}$ to $X_{P_{\mathcal{G}}}$. If $\mathcal{G}$ contains $[n]$, the combinatorial conditions on the building set $\mathcal{G}$ guarantee that $W_{L, \mathcal{G}}$ is smooth and can be described as an iterated blow-up of $\mathbb{P}L$ \cite[Section 4.1]{dCP95}. 

By Theorem~\ref{thm:vanish}\ref{item:push} and the projection formula, if $\mathcal{L}$ is any line bundle on $W_{L, \mathcal{G}}$, then the cohomology does not change when we pull back $\mathcal{L}$ along the map from $W_L$ to $W_{L, \mathcal{G}}$. In particular, Theorem~\ref{thm:vanish} implies the following result. 

\begin{theorem}\label{thm:building}
Let $L \subseteq \kk^n$ be a linear subspace which is not contained in any coordinate hyperplane, and let $\mathcal{G}$ be a building set on the lattice of flats of the matroid $\mathrm{M}$. 
Let $\mathcal{L}$ be the restriction of a nef line bundle on $X_{P_{\mathcal{G}}}$ to $W_{L, \mathcal{G}}$ of numerical dimension $d$. Then $H^i(W_{L, \mathcal{G}}, \mathcal{L}^{\otimes a}) = 0$ unless either $a \ge 0$ and $i = 0$, or $a < 0$ and $i = d$, the restriction map $H^0(X_{P_{\mathcal{G}}}, \mathcal{L}^{\otimes a}) \to H^0(W_{L, \mathcal{G}}, \mathcal{L}^{\otimes a})$ is surjective for any $a \ge 0$, and the ring $\bigoplus_{a \ge 0} H^0(W_{L, \mathcal{G}}, \mathcal{L}^{\otimes a})$ is Cohen--Macaulay and generated in degree $1$. 
\end{theorem}

In particular, Theorem~\ref{thm:building} proves vanishing theorems for line bundles on $\overline{M}_{0,n}$ \cite[Section 4.2]{dCP95} and wonderful compactifications of subspace arrangements (see \cite{BergmanPoly}).

\medskip

In \cite{BHMPW22}, the authors introduced the \emph{augmented wonderful variety} $W_{L}^{\operatorname{aug}}$ of a linear subspace $L \subseteq \kk^n$. This variety is constructed as an iterated blow-up of the projective completion $\mathbb{P}(L \oplus \kk)$, and it can also be defined as the closure of $L$ in the stellahedral toric variety ${X}_{\operatorname{St_n}}$, a certain toric compactification of $\kk^n$. We use Theorem~\ref{thm:vanish} to deduce the following vanishing theorem for $W_{L}^{\operatorname{aug}}$. 

\begin{theorem}
Let $L \subseteq \kk^n$ be a linear subspace, and let $\mathcal{L}$ be the restriction of a nef line bundle from ${X}_{\operatorname{St_n}}$ to $W_{L}^{\operatorname{aug}}$ of numerical dimension $d$. Then $H^i(W_{L}^{\operatorname{aug}}, \mathcal{L}^{\otimes a}) = 0$ unless either $a \ge 0$ and $i = 0$, or $a < 0$ and $i = d$, the restriction map $H^0(X_{\operatorname{St}_n}, \mathcal{L}^{\otimes a}) \to H^0(W_{L}^{\operatorname{aug}}, \mathcal{L}^{\otimes a})$ is surjective for any $a \ge 0$, and the ring $\bigoplus_{a \ge 0} H^0(W_{L}^{\operatorname{aug}}, \mathcal{L}^{\otimes a})$ is Cohen--Macaulay and generated in degree $1$. 
\end{theorem}

\begin{proof}
By replacing $\kk^n$ with the smallest coordinate subspace containing $L$, we can reduce to the case when $L$ is not contained in any coordinate hyperplane.  
As described in \cite[Remark 4.1]{BHMPW22}, $W_L^{\operatorname{aug}}$ arises as a stage in the iterated blow-up construction of the wonderful variety of $L \oplus \kk \subseteq \kk^{[n] \cup \{0\}}$. Set 
$$P = \sum_{S \ni 0} \Delta_S \subset \mathbb{R}^{[n] \cup \{0\}},$$
where the sum denotes Minkowski sum. Then the stellahedral toric variety $X_{\operatorname{St}_n}$ is isomorphic to $X_P$. Let $f$ denote the map from $X_{[n+1]}$ to $X_P$, so $W_{L}^{\operatorname{aug}}$ is isomorphic to $f(W_{L \oplus \kk})$. The result then follows from Theorem~\ref{thm:vanish} by pulling back to $W_{L \oplus \kk}$. 
\end{proof}

One can similarly deduce an analogue of Theorem~\ref{thm:Macaulay} for augmented $K$-rings of matroids \cite{LLPP24}. 

\bibliography{EFL_sources}
\bibliographystyle{alpha}

\end{document}